\documentclass{article}
\usepackage{amssymb,url}
\usepackage{amsthm}
\usepackage{mathrsfs}
\usepackage{amsmath}
\usepackage{graphicx}
\usepackage{caption}
\usepackage{subcaption}
\usepackage[textwidth=16cm,textheight=24cm]{geometry}

\newtheorem{lemma}{Lemma}[section]
\newtheorem{theorem}{Theorem}[section]
\newtheorem{corollary}{Corollary}[section]
\newtheorem{remark}{Remark}[section]
\newtheorem{assumptionA}{A-\hspace{-1.2mm}}
\newtheorem{assumptionB}{B-\hspace{-1.2mm}}
\newtheorem{assumptionC}{C-\hspace{-1.2mm}}

\begin{document}
\title{On Tamed Euler  Approximations of SDEs Driven by L\'evy Noise with Applications to Delay Equations}

\author{Konstantinos Dareiotis, Chaman Kumar and Sotirios Sabanis \\ School of Mathematics,  University of Edinburgh\\ Edinburgh, EH9 3JZ, United Kingdom}

 \maketitle

\begin{abstract}
We extend the taming techniques for explicit Euler approximations of stochastic differential equations (SDEs) driven by L\'evy noise with super-linearly growing drift coefficients. Strong convergence results are presented for the case of locally Lipschitz coefficients. Moreover, rate of convergence results are obtained in agreement with classical literature when the local Lipschitz continuity assumptions are replaced by global and, in addition, the drift coefficients satisfy polynomial Lipschitz continuity. Finally, we further extend these techniques to the case of  delay equations.
\end{abstract}

\section{Introduction}

In economics, finance, medical sciences, ecology, engineering,  and many other branches of sciences, one often encounters problems which are influenced by event-driven uncertainties. For example, in finance, the unpredictable nature of important events such market crashes, announcements made by central banks, changes in credit ratings, defaults, etc. might have sudden and significant impacts on the stock price movements. Stochastic differential equations (SDEs) with jumps, or more precisely SDEs driven by L\'evy noise, have been widely used to model such event-driven phenomena. The interested reader may refer, for example, to  \cite{cont2004, oksendal2005, situ2005} and references therein.

Many such SDEs do not have explicit solutions and therefore one requires numerical schemes so as to approximate their solutions. Over the past few years, several explicit and implicit schemes of SDEs driven by L\'evy noise have been studied and results on their strong and weak convergence were proved. For a comprehensive discussion on these schemes, one could refer to \cite{bruti2007, higham2005, higham2006, jacod2005, platen2010}, and references therein.

It is also known, however, that the computationally efficient explicit Euler schemes of SDEs (even without jumps) may not convergence in strong ($\mathcal{L}^q$) sense when the drift coefficients are allowed to grow super-linearly, see for example  \cite{hutzenthaler2010}. The development of tamed Euler schemes was a recent breakthrough in order to address this problem; one may consult  \cite{hutzenthaler2012, sabanis2013} as well as \cite{HJ, Sabanis2013a, Tretyakov-Zhang} and references therein for a thorough investigation of the subject.

In this article, we propose explicit tamed Euler schemes to numerically solve SDEs with random coefficients driven by L\'evy noise. The taming techniques developed here allow one to approximate these SDEs with drift coefficients that grow super-linearly. By adopting the approach of \cite{sabanis2013}, we prove   strong convergence in (uniform) $\mathcal{L}^q$ sense  of these tamed schemes by assuming one-sided local Lipschitz condition on drift and local Lipschitz conditions on both  diffusion and  jump coefficients. Moreover, our technical calculations are more refined than those of \cite{hutzenthaler2012, sabanis2013} in that we develop new techniques to overcome the challenges arising due to jumps. In addition, explicit formulations of the tamed Euler schemes are presented at the end of Section 3 for the case of SDEs driven by L\'evy noise which have non-random coefficients.

To the best of the authors' knowledge, the results obtained in this article are the first for the case of super-linear coefficients in this area. Moreover, the techniques developed here allow for further investigation of convergence properties of higher order explicit numerical schemes for SDEs driven by L\'evy noise with super-linear coefficients.

As an application of our approach which considers random coefficients, we also present in this article uniform $\mathcal{L}^q$ convergence results of explicit tamed Euler schemes for the case of stochastic delay differential equations (SDDEs) driven by L\'evy noise. The link between delay equations and random coefficients utilises ideas from \cite{gs2012}.  The aforementioned results are derived under the assumptions of one-sided local Lipschitz condition on drift and local Lipschitz conditions on both diffusion and jump coefficients with respect to non-delay variables, whereas these coefficients are only asked to be continuous with respect to arguments corresponding to delay variables.   It is worth mentioning here that our approach allows one to use our schemes to approximate SDDEs with jumps when drift coefficients can have super-linear growth in both delay and non-delay arguments.  Thus, the proposed tamed Euler schemes  provide significant improvements over the existing results available on numerical techniques of SDDEs, for example,  \cite{chenggui2013, kumar2014}. It should also be noted that, by adopting the approach of \cite{gs2012},  we prove the existence of a unique solution to the SDDEs driven by L\'evy noise under more relaxed conditions than those existing in the literature, for example,  \cite{jacob2009} whereby we ask for the local Lipschitz continuity only with respect to the non-delay variables.

Finally, rate of convergence results are obtained (which are in agreement with classical literature) when the local Lipschitz continuity assumptions are replaced by global and, in addition, the drift coefficients satisfy polynomial Lipschitz continuity. Similar results are also obtained for delay equations when the following assumptions hold - (a) drift coefficients satisfy  one-sided Lipschitz and polynomial Lipschitz conditions in non-delay variables whereas polynomial Lipschitz conditions in delay variables and (b) diffusion and jump coefficients satisfy Lipschitz conditions in non-delay variables whereas polynomial Lipschitz conditions in delay variables. This finding is itself a significant improvement over recent results in the area, see for example  \cite{chenggui2013} and references therein.

We conclude this section by introducing some basic notation. For a vector  $x\in \mathbb{R}^d$, we write $|x|$ for its Euclidean norm and for a $d\times m$ matrix $\sigma$, we write $|\sigma|$ for its Hilbert-Schmidt norm and $\sigma^*$ for its transpose. Also for $x,y \in \mathbb{R}^d$, $xy$ denotes the inner product of these two vectors. Further, the indicator function of a set $A$ is denoted by $I_A$, whereas $[x]$ stands for the integer part of a real number $x$. Let $\mathscr{P}$ be the predictable sigma-algebra on $\Omega \times \mathbb{R}_+ $ and $\mathscr{B}(V)$, the sigma-algebra of Borel sets of a topological space $V$. Also, let $T>0$ be  fixed and $\mathbb{L}^p$ denote the set of non-negative measurable functions $g$ on $[0,T]$, such that $\int_0^T|g_t|^pdt < \infty$. Finally, for a random variable $X$, the notation $X \in \mathcal{L}^p$ means $E|X|^p < \infty$.

\section{SDE with Random Coefficients Driven by L\'evy Noise} Let us assume that  $(\Omega, \{\mathscr{F}_t\}_{t \geq 0}, \mathscr{F}, P)$  denotes a probability space equipped with a filtration $\{\mathscr{F}_t\}_{t \geq 0}$ which is assumed to satisfy the usual conditions, i.e. $\mathscr{F}_0$ contains all $P$-null sets  and the filtration is  right continuous. Let  $w$ be an $\mathbb{R}^m-$valued  standard Wiener process. Further assume that $(Z, \mathscr{Z}, \nu)$ is a $\sigma-$finite measure space and  $N(dt,dz)$ is a Poisson random measure defined on $(Z, \mathscr{Z}, \nu)$ with intensity $\nu \not\equiv 0$ (in case $\nu \equiv 0$, one could consult \cite{sabanis2013}). Also let the compensated poisson random measure be denoted by $\tilde N(dt,dz):=N(dt,dz)-\nu(dz)dt$.

Let $b_t(x)$ and $\sigma_t(x)$ be  $\mathscr P \otimes \mathscr{B}(\mathbb R^d)$-measurable functions which respectively take values in $\mathbb R^d$ and $\mathbb R^{d \times m}$. Further assume that $\gamma_t(x,z)$ is  $\mathscr P \otimes \mathscr{B}(\mathbb R^d)\otimes \mathscr{Z}$-measurable function which takes values in $\mathbb R^{d}$. Also assume that $t_0$ and $t_1$ are fixed  constants satisfying $0 \leq t_0 <t_1 \leq T$.

We consider the following SDE
\begin{align}     \label{eq:sdewrc}
dx_t=b_t(x_t)dt+\sigma_t(x_t)dw_t+\int_{Z}\gamma_t(x_t,z) \tilde N(dt,dz)
\end{align}
almost surely for any $t \in [t_0,t_1]$ with initial value $x_{t_0}$ which is an $\mathscr{F}_{t_0}$-measurable random variable in $\mathbb R^d$.

\begin{remark}
For notational convenience, we write $x_t$ instead of $x_{t-}$ on the right hand side of the above equation. This does not cause any problem since the compensators of the martingales driving the equation are continuous. This notational convention shall be adopted throughout this article.
\end{remark}
\begin{remark}
In this article, we use $K>0$ to denote a generic constant which varies at different occurrences.
\end{remark}

The proof for the following lemma can be found in \cite{Mik}.
\begin{lemma}\label{lem:maximal:inequality}
Let $r \geq 2$. There exists a constant $K$, depending only on $r$, such that for every real-valued, $\mathscr{P} \otimes \mathscr Z-$measurable function $g$ satisfying
$$
\int_0^T\int_Z |g_t(z)|^2\nu(dz) dt < \infty
$$
almost surely, the following estimate holds,
\begin{align}
E\sup_{0 \leq t \leq T} \Big|\int_0^t \int_Z g_s(z)& \tilde{N}(ds,dz)\Big|^r \leq K E\Big(\int_0^T \int_Z|g_t(z)|^2\nu(dz)dt\Big)^{r/2} +K E\int_0^T \int_Z |g_t(z)|^r \nu(dz)dt. \label{eq:aobe}
\end{align}
It is  known  that if $1\leq r \leq 2$, then the second term in \eqref{eq:aobe} can be dropped.
\end{lemma}

\subsection{Existence and Uniqueness} \label{sec:exis:uniq:rc}
Let  $\mathcal A$ denotes the class of non-negative  predictable processes $L:=(L_t)_{t \in [0,T]}$ such that
$$
\int_0^T L_t dt < \infty
$$
for almost every $\omega \in \Omega$.

For the purpose of this section, the set of assumptions are listed  below.

\begin{assumptionA} \label{as:growth:rc:eu}
There exists an $\mathcal{M} \in \mathcal A$ such that
\begin{align}
 x b_t(x)+|\sigma_t(x)|^2 + \int_Z|\gamma_t(x, z)|^2 \nu(dz)& \leq \mathcal{M}_t(1+|x|^2) \notag
\end{align}
almost surely for any $t \in [t_0, t_1]$ and $x \in \mathbb{R}^d$.
\end{assumptionA}

\begin{assumptionA} \label{as:locallip:rc:eu}
For every $R>0$, there exists an $\mathcal{M}(R) \in \mathcal{A}$ such that,
\begin{align}
  (x-\bar x)  \, (b_t(x)-b_t(\bar x))+ |\sigma_t(x)-\sigma_t(\bar x)|^2   & +\int_Z |\gamma_t(x,z)-\gamma_t(\bar x,z)|^2\nu(dz)  \leq \mathcal{M}_t(R)|x-\bar x|^2 \notag
\end{align}
almost surely for any $t \in [t_0, t_1]$ whenever $|x|,|\bar x| \leq R$.
\end{assumptionA}

\begin{assumptionA}  \label{as:continuity:rc:eu}
For any $t \in [t_0,t_1]$ and $\omega \in \Omega$, the function $b_t(x)$ is  continuous in $x\in \mathbb{R}^d$.
\end{assumptionA}

The proof for the following theorem can be found in \cite{gyongy1980}.
\begin{theorem}        \label{thm:eu:rc}
Let Assumptions  A-1 to A-3 be satisfied. Then, there exists a unique solution to SDE \eqref{eq:sdewrc}.
\end{theorem}

\subsection{Moment Bounds}  \label{sec:mb:rc}
We make the following assumptions on the coefficients of  SDE \eqref{eq:sdewrc}.
\begin{assumptionA} \label{as:initial:rc:mb}
For a fixed $p \geq 2$,  $E|x_{t_0}|^p < \infty$.
\end{assumptionA}
\begin{assumptionA} \label{as:growth:rc:mb}
There exist a constant $L>0$ and a non-negative random variable $M$ satisfying $EM^\frac{p}{2}< \infty$ such that
\begin{align}
  x b_t(x) \vee |\sigma_t(x)|^2 \vee \int_Z|\gamma_t(x, z)|^2 \nu(dz)\leq L(M+|x|^2) \notag
\end{align}
almost surely for any $t \in [t_0, t_1]$ and $x\in \mathbb{R}^d$.
\end{assumptionA}
\begin{assumptionA}   \label{as:growth:p:rc:mb}
There exist a constant $L>0$ and a non-negative random variable $M'$ satisfying $EM' < \infty$ such that
$$
\int_Z |\gamma_t(x,z)|^p \nu(dz) \leq L(M'+|x|^p)
$$
almost surely for any  $t \in [t_0,t_1]$ and $x \in \mathbb{R}^d$.
\end{assumptionA}

The following is probably well-known. However, the proof is provided for the sake of completeness and for the justification of finiteness of the right hand side when applying Gronwall's lemma, something that is missing from the existing literature.
\begin{lemma} \label{lem:mb:rc}
Let Assumptions A-2 to   A-6 be satisfied. Then there exists a unique solution $(x_t)_{t\in [t_0,t_1]}$ of SDE \eqref{eq:sdewrc} and the following estimate holds
$$
E\sup_{t_0 \leq t \leq t_1}|x_t|^p \leq K,
$$
with $K:=K(t_0, t_1, L, p, E|x_{t_0}|^p, EM^\frac{p}{2}, EM')$.
\end{lemma}
\begin{proof}
The existence and uniqueness of solution to SDE \eqref{eq:sdewrc} follows immediately from Theorem \ref{thm:eu:rc} by noting that due to Assumption A-5, Assumption A-1 is satisfied.

Let us first define the stopping time $\pi_R:= \inf\{t \geq t_0: |x_t| >R\}\wedge t_1$, and notice that $|x_{t-}| \leq R$ for $ t_0 \leq t \leq \pi_R$.  By  It\^{o}'s formula,

\begin{align} \label{eq:ito:rc}
|x_t|^p &= |x_{t_0}|^p+ p \int_{t_0}^{t} |x_s|^{p-2} x_s b_s( x_s) ds + p\int_{t_0}^{t} |x_s|^{p-2} x_s \sigma_s( x_s)  dw_s  \notag
\\
& + \frac{p(p-2)}{2} \int_{t_0}^{t} |x_s|^{p-4}|\sigma_s^{*}(x_s) x_s|^2ds +\frac{p}{2}\int_{t_0}^{t} |x_s|^{p-2}|\sigma_s(x_s)|^2 ds  \notag
\\
&+  p\int_{t_0}^{t} \int_{Z} |x_s|^{p-2} x_{s} \gamma_s( x_{s},z)    \tilde N(ds,dz) \notag
\\
+\int_{t_0}^{t} &\int_{Z}\{ |x_{s}+\gamma_s( x_{s},z)|^p-|x_{s}|^p-p|x_{s}|^{p-2} x_{s}\gamma_s( x_{s},z) \}N(ds,dz)
\end{align}
almost surely for any $t \in [t_0,t_1]$. By virtue of Assumption A-5 and Young's inequality, one can estimate the second, fourth and fifth terms of equation \eqref{eq:ito:rc} by
\begin{align} \label{eq:2+4+5}
K M^\frac{p}{2}+K \int_{t_0}^{t} |x_s|^{p} ds.
\end{align}
Further, since the map $y \to |y|^p$ is of class $C^2$, by the formula for the remainder, for any $y_1, y_2 \in \mathbb{R}^d$, one gets
\begin{align}
|y_1+y_2|^p-|y_1|^p-p|y_1|^{p-2}y_1y_2& \leq K\int_0^1| y_1+q y_2|^{p-2}|y_2|^2 dq  \notag
\\
& \leq K(|y_1|^{p-2}|y_2|^2+|y_2|^p). \label{eq:y1y2}
\end{align}
Hence the last term of \eqref{eq:ito:rc} can be estimated by
\begin{align} \label{eq:last}
K\int_{t_0}^t \int_Z \{|x_{s}|^{p-2}|\gamma_s( x_{s},z)|^2+|\gamma_s( x_{s},z)|^p \} \ N(ds,dz).
\end{align}
One substitutes the estimates from \eqref{eq:2+4+5} and \eqref{eq:last} in equation \eqref{eq:ito:rc} which by taking suprema over $[t_0,u \wedge \pi_R]$ for $u \in [t_0, t_1]$ and expectations gives
\begin{align} \label{eq:C1+C5}
& E\sup_{t_0 \leq t \leq u \wedge \pi_R }|x_t|^p  \leq  E|x_{t_0}|^p+ K EM^\frac{p}{2}+ K E\int_{t_0}^{u \wedge \pi_R}|x_s|^p ds \notag
\\
&+  pE\sup_{t_0 \leq t \leq  u \wedge \pi_R }\Big|\int_{t_0}^{t} |x_s|^{p-2} x_s \sigma_s( x_s)  dw_s\Big|  \notag
\\
& + pE\sup_{t_0 \leq t \leq  u \wedge \pi_R }\Big|\int_{t_0}^{t} \int_{Z} |x_s|^{p-2} x_{s} \gamma_s(x_{s},z)   \tilde N(ds,dz)\Big| \notag
\\
&+  KE\int_{t_0}^{u \wedge \pi_R} \int_Z  \{|x_{s}|^{p-2} |\gamma_s( x_{s},z)|^2+|\gamma_s( x_{s},z)|^p \}\ N(ds,dz) \notag
\\
&=:C_1+C_2+C_3+C_4+C_5.
\end{align}
Here $C_1:=E|x_{t_0}|^p+K EM^\frac{p}{2}$.  By the Burkholder-Davis-Gundy inequality, $C_3$ can be estimated as
\begin{align*}
C_3 & = pE\sup_{t_0 \leq t < u \wedge \pi_R }\Big|\int_{t_0}^{t} |x_{s-}|^{p-2} x_{s-} \sigma_s( x_{s-})  dw_s\Big|
\\
& \leq KE\sup_{t_0 \leq t \leq  u\wedge \pi_R}|x_{t-}|^{p-1}\left(\int_{t_0}^{u \wedge \pi_R} |\sigma_s(x_{s-})|^2ds\right)^{1/2}
\end{align*}
which on the application of Young's inequality gives
\begin{align*}
C_3  \leq \frac{1}{4} E\sup_{t_0 \leq t \leq  u\wedge \pi_R}|x_{t-}|^p+K E\left(\int_{t_0}^{u \wedge \pi_R} |\sigma_s(x_s)|^2ds\right)^{p/2}
\end{align*}
and then due to H\"{o}lder's inequality and Assumption A-5, one has
\begin{align} \label{eq:C3}
C_3  \le \frac{1}{4} E\sup_{t_0 \leq t \leq  u\wedge \pi_R}|x_{t-}|^p+KEM^\frac{p}{2}+K E\int_{t_0}^{u \wedge \pi_R} |x_r|^pds< \infty.
\end{align}

To estimate $C_4$, one uses Lemma \ref{lem:maximal:inequality} to write
\begin{align*}
C_4 &:=pE\sup_{t_0 \leq t \leq  u \wedge \pi_R }\Big|\int_{t_0}^{t} \int_{Z} |x_{s-}|^{p-2} x_{s-} \gamma_s(x_{s-},z)   \tilde N(ds,dz)\Big|
\\
& \leq K E\Big(\int_{t_0}^{u \wedge \pi_R} \int_{Z} |x_{s-}|^{2p-2} |\gamma_s(x_{s-},z)|^2   \nu(dz) ds \Big)^\frac{1}{2}
\\
& \leq K E\sup_{t_0 \leq t \leq  u\wedge \pi_R}|x_{t-}|^{p-1}\Big(\int_{t_0}^{u \wedge \pi_R} \int_{Z}|\gamma_s(x_{s},z)|^2   \nu(dz) ds \Big)^\frac{1}{2}
\end{align*}
which due to Young's inequality, Assumption A-5 and  H\"{o}lder's inequality implies
\begin{align} \label{eq:C4}
C_4   \leq \frac{1}{4} E\sup_{t_0 \leq t \leq  u\wedge \pi_R}|x_{t-}|^p+KEM^\frac{p}{2}+K E\int_{t_0}^{u \wedge \pi_R} |x_r|^pds< \infty.
\end{align}
For $C_5$, by Assumptions A-5, A-6 and Young's inequality,
\begin{align} \label{eq:C5}
C_5&:= KE\int_{t_0}^{u \wedge \pi_R}  \int_Z \left(|x_{s}|^{p-2} |\gamma_s( x_s,z)|^2+|\gamma_s( x_s,z)|^p \right)\ \nu (dz)ds \notag
\\
& \leq  KE\int_{t_0}^{u \wedge \pi_R}   \left\{|x_{s}|^{p-2}(M+|x_{s}|^2)+M'+|x_s|^p \right\}ds \notag
\\
& \leq K EM^\frac{p}{2}+K EM' +EK\int_{t_0}^{u \wedge \pi_R} |x_r|^pds < \infty.
\end{align}
By substituting the estimates from \eqref{eq:C3}-\eqref{eq:C5} in \eqref{eq:C1+C5}, one has
\begin{align}                   \label{eq: before Gronwall}
E\sup_{t_0 \leq t \leq  u \wedge \pi_R }|x_t|^p  \leq K+\frac{1}{2} E\sup_{t_0 \leq t \leq  u\wedge \pi_R}|x_{t-}|^p+K E\int_{t_0}^{u \wedge \pi_R} |x_r|^pds  < \infty
\end{align}
for any $u \in [t_0, t_1]$. In particular we obtain
$$
E\sup_{t_0 \leq t \leq  t_1 \wedge \pi_R }|x_t|^p< \infty.
$$
Since it holds that
$$
E\sup_{t_0 \leq t \leq  u \wedge \pi_R }|x_{t-}|^p \leq E\sup_{t_0 \leq t \leq  u \wedge \pi_R }|x_t|^p,
$$
by rearrenging in \eqref{eq: before Gronwall}, we obtain
\begin{align} \nonumber
E\sup_{t_0 \leq t \leq  u \wedge \pi_R }|x_t|^p  &\leq K+K E\int_{t_0}^{u \wedge \pi_R} |x_r|^pds  \\
& \leq K +E\int_{t_0}^u \sup_{t_0 \leq t \leq  s \wedge \pi_R }|x_t|^pds  < \infty.
\end{align}
From here we can finish the proof by Gronwall's and Fatou's lemmas.
\end{proof}
\section{Tamed Euler Scheme}
For every $n \in  \mathbb{N}$, let $b_t^n(x)$ and $\sigma_t^n(x)$ are  $\mathscr P \otimes \mathscr{B}(\mathbb R^d)$-measurable functions which respectively take values in $\mathbb R^d$ and $\mathbb R^{d \times m}$. Also, for every $n \in \mathbb{N}$, let $\gamma_t^n(x,z)$ be $\mathscr P \otimes \mathscr{B}(\mathbb R^d)\otimes \mathscr{Z}$-measurable function which takes values in $\mathbb R^{d}$. For every $n \in \mathbb{N}$, we consider a scheme of SDE \eqref{eq:sdewrc} as defined below,
\begin{align} \label{eq:em:sdewrc}
dx_t^n=b_t^n(x^n_{\kappa(n,t)})dt + \sigma_t^n(x^n_{\kappa(n,t)})dw_t+\int_{Z}\gamma_t^n(x^n_{\kappa(n,t)},z)\tilde N(dt, dz), \,
\end{align}
almost surely for any $t \in [t_0, t_1]$ where the initial value $x_{t_0}^n$ is an $\mathscr{F}_{t_0}$-measurable random variable which takes values in $\mathbb{R}^d$ and  function  $\kappa$ is  defined by
\begin{align} \label{eq:kappa:sdewrc}
\kappa(n, t):=\frac{[n(t-t_0)]}{n}+t_0
\end{align}
for any $t \in [t_0, t_1]$.
\subsection{Moment Bounds}   \label{sec:mb:em:rc}
We make the following assumptions on the coefficients of the scheme \eqref{eq:em:sdewrc}.
\begin{assumptionB} \label{as:initial:em:mb}
We have $\sup_{n \in \mathbb{N}}E|x^n_{t_0}|^p < \infty$.
\end{assumptionB}
\begin{assumptionB} \label{as:growth:tem:mb}
There exist a constant $L>0$ and a sequence $(M_n)_{n \in \mathbb{N}}$ of non-negative random variables satisfying $\sup_{n \in \mathbb{N}} EM_n^\frac{p}{2} < \infty$ such that
\begin{align}
  xb_t^n(x) \vee |\sigma_t^n(x)|^2 \vee \int_Z|\gamma_t^n(x, z)|^2 \nu(dz)\leq L(M_n+|x|^2) \notag
\end{align}
almost surely for any $t \in [t_0, t_1]$, $n \in \mathbb{N}$ and $x\in \mathbb{R}^d$.
\end{assumptionB}
\begin{assumptionB}   \label{as:growth:p:em:mb}
There exist a constant $L>0$ and a sequence $(M_n')_{n \in \mathbb{N}}$ of non-negative random variables satisfying $\sup_{n \in \mathbb{N}} EM_n' < \infty$  such that
$$
\int_Z |\gamma_t^n(x,z)|^p \nu(dz) \leq L(M_n'+|x|^p)
$$
almost surely for any  $t \in [t_0,t_1]$, $n \in \mathbb{N}$ and $x \in \mathbb{R}^d$.
\end{assumptionB}
Below is our taming assumption on drift coefficient of  scheme \eqref{eq:em:sdewrc} following the approach of \cite{sabanis2013}.
\begin{assumptionB} \label{as:bn}
For any $t \in [t_0, t_1]$ and $x \in \mathbb{R}^d$,
$$
|b_t^n(x)| \leq n^\theta
$$
almost surely with $\theta\in (0,\frac{1}{2}]$ for every  $n \in \mathbb{N}$.
\end{assumptionB}
\begin{remark}
Note that due to assumption B-4, for each $n \ge 1$, the norm of
$b^n$ is a bounded function of $t$ and $x$ which along with B-1 and B-2
guarantee the existence of a unique solution to \eqref{eq:em:sdewrc}.
Moreover, they also guarantee that for each $n\ge 1$,
\begin{equation*}
E\sup_{0\le t\le T}|x^n_t|^p < \infty.
\end{equation*}
 Clearly, one cannot claim at this point that this bound is independent of $n$. Nevertheless, as a result of this observation, one needs not to apply stopping time arguments, similar to the one used in the proof of Lemma \ref{lem:mb:rc}, in the proofs of Lemma \ref{lem:tem:conv:onestep} and Lemma  \ref{lem:mb:tes} mentioned below.
\end{remark}
\begin{lemma} \label{lem:tem:conv:onestep}
Let Assumptions  B-2 to B-4 hold. Then
\begin{align*}
\int_{t_0}^{u}E|x^n_t-x^n_{\kappa(n,t)}|^p dt & \leq  K n^{-1} + K n^{-1}  \int_{t_0}^{u}  E|x^n_{\kappa(n,t)}|^p dt
\end{align*}
for any $u \in [t_0, t_1]$ with  $K:=K\big(t_0,t_1,L,p,\sup_{n \in \mathbb{N}}EM_n^\frac{p}{2}, \sup_{n \in \mathbb{N}}EM_n'\big)$ which does not depend on $n$.
\end{lemma}
\begin{proof} From the definition of  scheme \eqref{eq:em:sdewrc}, one writes,
\begin{align*}
E&|x^n_t-x^n_{\kappa(n,t)}|^p  \leq K E \Big|\int_{\kappa(n,t)}^{t}  b_s^n(x^n_{\kappa(n,s)}) ds \Big|^p+ K E \Big|\int_{\kappa(n,t)}^{t}   \sigma^{n}_s(x^n_{\kappa(n,s)})  dw_s \Big|^p
\\
& \qquad \qquad+ KE\Big|\int_{\kappa(n,t)}^{t}  \int_{Z}   \gamma_s^n(x^n_{\kappa(n,s)},z)   \tilde N(ds,dz)\Big|^p.
\end{align*}
which on the application of H\"{o}lder's inequality and an elementary stochastic inequalities gives
\begin{align*}
E|x^n_{t}-x^n_{\kappa(n,t)}|^p & \leq   Kn^{-(p-1)} E \int_{\kappa(n,t)}^{t}  |b_s^n(x^n_{\kappa(n,s)})|^p ds + K E \Big(\int_{\kappa(n,t)}^{t}   |\sigma^{n}_s(x^n_{\kappa(n,s)}) |^2 ds\Big)^\frac{p}{2}
\\
+   K  E\Big(\int_{\kappa(n,t)}^{t} & \int_{Z}  | \gamma_s^n(x^n_{\kappa(n,s)},z)|^2 \nu(dz) ds\Big)^\frac{p}{2} +  K E\int_{\kappa(n,t)}^{t}  \int_{Z}  | \gamma_s^n(x^n_{\kappa(n,s)},z)|^p \nu(dz) ds.
\end{align*}
On using Assumptions  B-2, B-3 and  B-4, one obtains,
\begin{align*}
E|x^n_{t}-x^n_{\kappa(n,t)}|^p & \leq   K\Big(n^{-p(1-\theta)} +  n^{-\frac{p}{2}}E  (M_n+|x^n_{\kappa(n,t)}|^2)^\frac{p}{2} +  n^{-1}E  (M_n'+|x^n_{\kappa(n,t)}|^p)\Big)
\end{align*}
which completes the proof by noticing that $\theta \in (0,\frac{1}{2}]$ and $p\geq 2$.
\end{proof}
\begin{lemma} \label{lem:mb:tes}
Let Assumptions B-1 to B-4 be satisfied. Then,
$$
\sup_{n \in \mathbb{N}} E\sup_{t_0 \leq t \leq t_1}|x_t^n|^p \leq K,
$$
with $K:=K\big(t_0,t_1,L,p, \sup_{n\in \mathbb{N}} E|x_{t_0}^n|^p, \sup_{n\in \mathbb{N}}EM_n^\frac{p}{2},\sup_{n\in \mathbb{N}}EM_n' \big)$ which is independent of $n$.
\end{lemma}
\begin{proof}
By the application of It\^{o} formula, one gets
\begin{align}
|x_t^n|^p &= |x_{t_0}^n|^p+ p \int_{t_0}^{t} |x_s^n|^{p-2} x_s^n b_s^n( x_{\kappa(n,s)}^n) ds + p\int_{t_0}^{t} |x_s^n|^{p-2} x_s^n \sigma_s^n(x_{\kappa(n,s)}^n)  dw_s  \notag
\\
& + \frac{p(p-2)}{2} \int_{t_0}^{t} |x_s^n|^{p-4}|\sigma_s^{n*}(x_{\kappa(n,s)}^n) x_s^n|^2ds +\frac{p}{2}\int_{t_0}^{t} |x_s^n|^{p-2}|\sigma_s^n(x_{\kappa(n,s)}^n)|^2 ds  \notag
\\
&+  p\int_{t_0}^{t} \int_{Z} |x_s^n|^{p-2} x_{s}^n \gamma_s^n( x_{\kappa(n,s)}^n,z)    \tilde N(ds,dz) \notag
\\
+\int_{t_0}^{t} &\int_{Z}\{ |x_{s}^n+\gamma_s^n(x_{\kappa(n,s)}^n,z)|^p-|x_{s}^n|^p-p|x_{s}^n|^{p-2} x_{s}^n\gamma_s^n(x_{\kappa(n,s)}^n,z) \}N(ds,dz) \label{eq:tem:ito}
\end{align}
almost surely for any $t \in [t_0,t_1]$. In order to estimate second term of \eqref{eq:tem:ito}, one writes
\begin{align*}
x_s^n b_s^n( x_{\kappa(n,s)}^n) =(x_s^n-x_{\kappa(n,s)}^n) b_s^n( x_{\kappa(n,s)}^n) + x_{\kappa(n,s)}^n b_s^n( x_{\kappa(n,s)}^n)
\end{align*}
which due to Assumption B-2 and equation \eqref{eq:em:sdewrc} gives
\begin{align*}
x_s^n b_s^n( x_{\kappa(n,s)}^n)  \leq & |b_s^n( x_{\kappa(n,s)}^n)|\Big\{\Big|\int^{s}_{\kappa(n,s)}b_r^n(x^n_{\kappa(n,r)})dr \Big|
 +  \Big|\int^{s}_{\kappa(n,s)}\sigma_r^n(x^n_{\kappa(n,r)})dw_r\Big|
 \\
 &   +\Big|\int^{s}_{\kappa(n,s)}\int_{Z}\gamma_r^n(x^n_{\kappa(n,r)},z)\tilde N(dr, dz)\Big|\Big\} + K(M_n+|x_{\kappa(n,s)}^n|^2)
\end{align*}
and then Assumption B-4 implies,
\begin{align*}
x_s^n b_s^n( x_{\kappa(n,s)}^n)  \leq & n^{2\theta-1}
 + n^\theta \Big|\int^{s}_{\kappa(n,s)}\sigma_r^n(x^n_{\kappa(n,r)})dw_r\Big| \notag
 \\
 &+n^\theta \Big|\int^{s}_{\kappa(n,s)}\int_{Z}\gamma_r^n(x^n_{\kappa(n,r)},z)\tilde N(dr, dz)\Big| + K(M_n+|x_{\kappa(n,s)}^n|^2)
\end{align*}
almost surely for any $s \in [t_0,t_1]$. By using the fact that $\theta \in (0,1/2]$ implies $2\theta-1 \leq 0$, one obtains
\begin{align*}
|x_s^n|^{p-2} & x_s^n b_s^n( x_{\kappa(n,s)}^n)  \leq
|x_s^n|^{p-2} + n^\theta |x_s^n|^{p-2} \Big|\int^{s}_{\kappa(n,s)}\sigma_r^n(x^n_{\kappa(n,r)})dw_r\Big|
 \\
 &  \,\,\, +n^\theta |x_s^n|^{p-2}\Big|\int^{s}_{\kappa(n,s)}\int_{Z}\gamma_r^n(x^n_{\kappa(n,r)},z)\tilde N(dr, dz)\Big| + K |x_s^n|^{p-2}(M_n+|x_{\kappa(n,s)}^n|^2)
\end{align*}
which on using Young's inequality along with the inequality $|x_s^n|^{p-2} \leq 2^{p-3} |x_s^n-x_{\kappa(n,s)}^n|^{p-2}+2^{p-3}|x_{\kappa(n,s)}^n|^{p-2}$ gives
\begin{align}
 |x_s^n|^{p-2}  x_s^n b_s^n( x_{\kappa(n,s)}^n) & \leq
1+ K |x_s^n|^p + K n^{\theta\frac{p}{2}}  \Big|\int^{s}_{\kappa(n,s)}\sigma_r^n(x^n_{\kappa(n,r)})dw_r\Big|^\frac{p}{2} \notag
 \\
 &  + K n^{\theta} |x_{\kappa(n,s)}^n|^{p-2}\Big|\int^{s}_{\kappa(n,s)}\int_{Z}\gamma_r^n(x^n_{\kappa(n,r)},z)\tilde N(dr, dz)\Big| \notag
 \\
& + K n^{\theta} |x_s^n-x_{\kappa(n,s)}^n|^{p-2}\Big|\int^{s}_{\kappa(n,s)}\int_{Z}\gamma_r^n(x^n_{\kappa(n,r)},z)\tilde N(dr, dz)\Big|  \notag
\\
& + K (M_n^{\frac{p}{2}}+|x_{\kappa(n,s)}^n|^p) \label{eq:new}
\end{align}
almost surely for any $s \in [t_0,t_1]$. Therefore, by substituting  estimates from equations \eqref{eq:y1y2} and \eqref{eq:new}  in equation \eqref{eq:tem:ito}, one obtains for $u \in [t_0,t_1]$,
\begin{align}
E \sup_{t_0 \leq t \leq u} &|x_t^n|^p   \leq E|x_{t_0}^n|^p+ K + K E \int_{t_0}^{u} |x_s^n|^p  ds \notag
\\
& + K n^{\theta\frac{p}{2}} E \int_{t_0}^{u}   \Big|\int^{s}_{\kappa(n,s)}\sigma_r^n(x^n_{\kappa(n,r)})dw_r\Big|^\frac{p}{2}  ds  \notag
\\
& +K n^{\theta} E\int_{t_0}^{u}  \Big|\int^{s}_{\kappa(n,s)}\int_{Z}|x_{\kappa(n,s)}^n|^{p-2} \gamma_r^n(x^n_{\kappa(n,r)},z)\tilde N(dr, dz)\Big| ds \notag
\\
&+  K n^{\theta} E \int_{t_0}^{u} |x_s^n-x_{\kappa(n,s)}^n|^{p-2}\Big|\int^{s}_{\kappa(n,s)}\int_{Z}\gamma_r^n(x^n_{\kappa(n,r)},z)\tilde N(dr, dz)\Big| ds \notag
\\
&+KE\int_{t_0}^{u} (M_n^{\frac{p}{2}}+|x_{\kappa(n,s)}^n|^p) ds+ pE \sup_{t_0 \leq t \leq u} \Big|\int_{t_0}^{t} |x_s^n|^{p-2} x_s^n \sigma_s^n(x_{\kappa(n,s)}^n)  dw_s \Big| \notag
\\
& + K E \int_{t_0}^{u} |x_s^n|^{p-2}|\sigma_s^{n}(x_{\kappa(n,s)}^n)|^2ds \notag
\\
& +  pE \sup_{t_0 \leq t \leq u}\Big|\int_{t_0}^{t} \int_{Z} |x_s^n|^{p-2} x_{s}^n \gamma_s^n( x_{\kappa(n,s)}^n,z)    \tilde N(ds,dz) \Big| \notag
\\
& +E \int_{t_0}^{u} \int_{Z}\{ |x_{s}^n|^{p-2}|\gamma_s^n(x_{\kappa(n,s)}^n,z)|^2+ |\gamma_s^n(x_{\kappa(n,s)}^n,z)|^p \}N(ds,dz) \notag
\\
& \quad=: E_1+E_2+E_3+E_4+E_5+E_6+E_7+E_8+E_9+E_{10}. \label{eq:E1+E10}
\end{align}
Here $E_1:=E|x_{t_0}^n|^p+K$. One estimates $E_2$ by
\begin{align} \label{eq:E2}
E_2 := K  E \int_{t_0}^{u} |x_s^n|^p  ds \leq K  \int_{t_0}^{u} E\sup_{t_0 \leq r \leq s}|x_r^n|^p  ds.
\end{align}
In order to estimate  $E_3$, one applies an elementary stochastic inequality to obtain
\begin{align*}
E_3 & := K n^{\theta\frac{p}{2}} E \int_{t_0}^{u}   \Big|\int^{s}_{\kappa(n,s)}\sigma_r^n(x^n_{\kappa(n,r)})dw_r\Big|^\frac{p}{2}  ds \notag
\\
& \leq K n^{\theta\frac{p}{2}}  \int_{t_0}^{u}  E \Big(\int^{s}_{\kappa(n,s)}|\sigma_r^n(x^n_{\kappa(n,r)})|^2dr\Big)^\frac{p}{4}  ds \notag
\end{align*}
and then on the application of Assumption B-2, one obtains
\begin{align*}
E_3 \leq K n^{\frac{p}{2}(\theta-\frac{1}{2})}  \int_{t_0}^{u}  E (M_n+|x^n_{\kappa(n,s)}|^2)^\frac{p}{4}  ds
\end{align*}
which by noticing that $\frac{p}{2}(\theta-\frac{1}{2}) \in (-\frac{p}{4},0]$ implies
\begin{align} \label{eq:E3}
E_3  \leq K\int_{t_0}^{u}  E \{1+(M_n+|x^n_{\kappa(n,s)}|^2)^\frac{p}{2}\}  ds   \leq K +\int_{t_0}^{u} E \sup_{t_0 \leq r \leq s} |x^n_{r}|^p  ds.
\end{align}
By Lemma \ref{lem:maximal:inequality}, one estimates $E_4$ as
\begin{align}
E_4 & :=K n^{\theta} E\int_{t_0}^{u}  \Big|\int^{s}_{\kappa(n,s)}\int_{Z}|x_{\kappa(n,s)}^n|^{p-2} \gamma_r^n(x^n_{\kappa(n,r)},z)\tilde N(dr, dz)\Big| ds \notag
\\
& \leq K n^{\theta} \int_{t_0}^{u}  E\Big(\int^{s}_{\kappa(n,s)}\int_{Z}|x_{\kappa(n,s)}^n|^{2p-4} |\gamma_r^n(x^n_{\kappa(n,r)},z)|^2 \nu(dz) dr \Big)^\frac{1}{2} ds \notag
\end{align}
which due to Assumption B-2 gives
\begin{align}
E_4 & \leq  KE\sup_{t_0 \leq s \leq u}|x_s^n|^{p-2}  n^{\theta-\frac{1}{2}} \int_{t_0}^{u}  (M_n+ |x^n_{\kappa(n,s)}|^2)^\frac{1}{2} ds \notag
\end{align}
and then on using Young's inequality and H\"{o}lder's inequality, one obtains
\begin{align}
E_4 \leq \frac{1}{8} E\sup_{t_0 \leq s \leq u}|x_s^n|^ p+ K  n^{\frac{p}{2}(\theta-\frac{1}{2})} E \int_{t_0}^{u}  (M_n+ |x^n_{\kappa(n,s)}|^2)^\frac{p}{4} ds. \notag
\end{align}
By noticing that $\theta \in (0,\frac{1}{2}]$, one has
\begin{align} \label{eq:E4}
E_4 &\leq \frac{1}{8} E\sup_{t_0 \leq s \leq u}|x_s^n|^ p+ K E  \int_{t_0}^{u}  \{1+(M_n+ |x^n_{\kappa(n,s)}|^2)^\frac{p}{2} \}ds. \notag
\\
& \leq \frac{1}{8} E\sup_{t_0 \leq s \leq u}|x_s^n|^ p+ K  +K \int_{t_0}^{u} E\sup_{t_0 \leq r \leq s}|x^n_r|^p  ds.
\end{align}
 Further, to estimate $E_5$, one uses Young's inequality and H\"older's inequality to write
\begin{align}
& E_5:=K n^{\theta} E \int_{t_0}^{u} |x_s^n-x_{\kappa(n,s)}^n|^{p-2}\Big|\int^{s}_{\kappa(n,s)}\int_{Z}\gamma_r^n(x^n_{\kappa(n,r)},z)\tilde N(dr, dz)\Big| ds \notag
\\
&\leq K   n^{\theta} \int_{t_0}^{u}E|x_s^n-x_{\kappa(n,s)}^n|^p  ds + K n^{\theta}\int_{t_0}^{u}E\Big|\int^{s}_{\kappa(n,s)}\int_{Z}\gamma_r^n(x^n_{\kappa(n,r)},z)\tilde N(dr, dz)\Big|^\frac{p}{2} ds \notag
\\
&\leq K   n^{\theta} \int_{t_0}^{u}E|x_s^n-x_{\kappa(n,s)}^n|^p  ds +1+K n^{2\theta}\int_{t_0}^{u}E\Big|\int^{s}_{\kappa(n,s)}\int_{Z}\gamma_r^n(x^n_{\kappa(n,r)},z)\tilde N(dr, dz)\Big|^p ds. \notag
\end{align}
which on the application of Lemma \ref{lem:maximal:inequality} and Lemma \ref{lem:tem:conv:onestep} implies,
\begin{align*}
E_5 &\leq  1+K n^{\theta-1}  + K n^{\theta-1}  \int_{t_0}^{u}  E|x^n_{\kappa(n,s)}|^p ds \notag
\\
& + K n^{2\theta}\int_{t_0}^{u}E\Big(\int^{s}_{\kappa(n,s)}\int_{Z}|\gamma_r^n(x^n_{\kappa(n,r)},z)|^2\nu(dz) dr\Big)^\frac{p}{2} ds
\\
& + K n^{2\theta}\int_{t_0}^{u}E\int^{s}_{\kappa(n,s)}\int_{Z}|\gamma_r^n(x^n_{\kappa(n,r)},z)|^p\nu(dz) dr ds.
\end{align*}
By using Assumptions B-2 and B-3, one obtains
\begin{align*}
E_5 &\leq  1+ K n^{\theta-1}  + K n^{\theta-1}  \int_{t_0}^{u}  E|x^n_{\kappa(n,s)}|^p ds + K n^{2\theta-\frac{p}{2}}\int_{t_0}^{u}E\big(M_n+|x^n_{\kappa(n,s)}|^2\big)^\frac{p}{2} ds
\\
& \qquad+ K n^{2\theta-1}\int_{t_0}^{u}E(M_n'+|x^n_{\kappa(n,s)}|^p) ds.
\end{align*}
Notice that $2\theta-1 \in (-1,0]$ and $p\geq 2$. Hence one has
\begin{align} \label{eq:E5}
E_5 &\leq  K+ KE M_n^\frac{p}{2} +EM_n'  + K   \int_{t_0}^{u}  E|x^n_{\kappa(n,s)}|^p ds  \leq K   + K \int_{t_0}^{u}  E\sup_{t_0 \leq r \leq s}|x^n_r|^p ds.
\end{align}
It is easy to observe that $E_6$ can be estimated by
\begin{align} \label{eq:E6}
E_6:=KE\int_{t_0}^{u} (M_n^{\frac{p}{2}}+|x_{\kappa(n,s)}^n|^p) ds \leq K   + K \int_{t_0}^{u}  E\sup_{t_0 \leq r \leq s}|x^n_r|^p ds.
\end{align}
By using Burkholder-Davis-Gundy inequality and Assumption B-2, one obtains the following  estimates of $E_7$,
\begin{align*}
E_7&:=pE \sup_{t_0 \leq t \leq u} \Big|\int_{t_0}^{t} |x_s^n|^{p-2} x_s^n \sigma_s^n(x_{\kappa(n,s)}^n)  dw_s \Big|
\\
& \leq K E \Big(\int_{t_0}^{u} |x_s^n|^{2p-2} | \sigma_s^n(x_{\kappa(n,s)}^n)|^2  ds \Big)^\frac{1}{2}   \leq K E \Big(\int_{t_0}^{u} |x_s^n|^{2p-2} (M_n+|x_{\kappa(n,s)}^n|^2)  ds \Big)^\frac{1}{2} \notag
\\
& \leq K E \sup_{t_0 \leq s \leq u}|x_s^n|^{p-1} \Big(\int_{t_0}^{u}  (M_n+|x_{\kappa(n,s)}^n|^2)  ds \Big)^\frac{1}{2}
\end{align*}
which due to Young's inequality and  H\"{o}lder's inequality gives
\begin{align} \label{eq:E7}
E_7   \leq \frac{1}{8}  E \sup_{t_0 \leq s \leq u}|x_s^n|^p+K+K E \int_{t_0}^{u}  E \sup_{t_0 \leq r \leq s}|x_r^n|^p  ds.
\end{align}
Similarly, by using Assumption B-2 and Young's inequality, $E_8$ can be estimated by
\begin{align} \label{eq:E8}
E_8 & := K E \int_{t_0}^{u} |x_s^n|^{p-2}|\sigma_s^{n}(x_{\kappa(n,s)}^n)|^2ds  \leq K+K  \int_{t_0}^{u} E\sup_{t_0 \leq r \leq s}|x_r^n|^p ds.
\end{align}
Further one uses Lemma \ref{lem:maximal:inequality} and Assumption B-2 to estimate $E_9$ by,
\begin{align*}
 E_9 &:= pE \sup_{t_0 \leq t \leq u}\Big|\int_{t_0}^{t} \int_{Z} |x_s^n|^{p-2} x_{s}^n \gamma_s^n( x_{\kappa(n,s)}^n,z)    \tilde N(ds,dz) \Big| \notag
 \\
 & \leq K E \Big(\int_{t_0}^{u} \int_{Z} |x_s^n|^{2p-2} |\gamma_s^n( x_{\kappa(n,s)}^n,z)|^2 \nu(dz) ds \Big)^\frac{1}{2}
 \\
 &\leq K E \Big(\int_{t_0}^{u}  |x_s^n|^{2p-2} (M_n+| x_{\kappa(n,s)}^n|^2) ds \Big)^\frac{1}{2} \notag
\end{align*}
which due to  Young's inequality and H\"{o}lder's inequality gives
\begin{align} \label{eq:E9}
 E_9   & \leq \frac{1}{8} E \sup_{t_0 \leq s \leq u}|x_s^n|^{p} + K + K E \int_{t_0}^{u} E \sup_{t_0 \leq r \leq s}|x_r^n|^{p} ds.
\end{align}
Finally, due to Assumptions B-2 and  B-3, $E_{10}$ can be estimated as follow,
\begin{align*}
E_{10}&:=E \int_{t_0}^{u} \int_{Z}\{ |x_{s}^n|^{p-2}|\gamma_s^n(x_{\kappa(n,s)}^n,z)|^2+ |\gamma_s^n(x_{\kappa(n,s)}^n,z)|^p \}N(ds,dz)
\\
& =E \int_{t_0}^{u} \int_{Z}\{ |x_{s}^n|^{p-2}|\gamma_s^n(x_{\kappa(n,s)}^n,z)|^2+ |\gamma_s^n(x_{\kappa(n,s)}^n,z)|^p \}\nu(dz) ds
\\
& =E \int_{t_0}^{u}  |x_{s}^n|^{p-2}(M_n+|x_{\kappa(n,s)}^n|^2) ds+ E \int_{t_0}^{u} (M_n'+|x_{\kappa(n,s)}^n|^p) ds
\end{align*}
and then Young's inequality implies
\begin{align} \label{eq:E10}
E_{10} \leq K+  \int_{t_0}^{u} E\sup_{t_0 \leq r \leq s}|x_{r}^n|^{p} ds.
\end{align}
By substituting estimates from \eqref{eq:E2} -
\eqref{eq:E10} in equation \eqref{eq:E1+E10}, one obtains
\begin{align*}
E \sup_{t_0 \leq t \leq u}|x_t^n|^{p} \leq  \frac{1}{2} E \sup_{t_0 \leq t \leq u}|x_t^n|^{p} +K + K E \int_{t_0}^{u} E \sup_{t_0 \leq r \leq s}|x_r^n|^{p} ds.
\end{align*}
The application of Gronwall's Lemma  completes the proof.
\end{proof}
\begin{remark}   \label{rem:growth:r:em:mb}
Due to Assumptions B-2 and B-3, there exist a constant $L>0$ and a sequence $(M_n')_{n \in \mathbb{N}}$ of non-negative random variables satisfying $\sup_{n \in \mathbb{N}} EM_n' < \infty$  such that
$$
\int_Z |\gamma_t^n(x,z)|^r \nu(dz) \leq L(M_n'+|x|^r)
$$
almost surely for any  $2 \leq r \leq p$, $t \in [t_0,t_1]$, $n \in \mathbb{N}$ and $x \in \mathbb{R}^d$.
\end{remark}
\begin{lemma} \label{lem:tem:new:conv:onestep}
Let Assumptions B-1 to B-4 be satisfied. Then
\begin{align*}
\sup_{t_0\leq t\leq t_1}E|x^n_t-x^n_{\kappa(n,t)}|^r  & \leq  K n^{-1}
\end{align*}
for any $2 \leq r \leq p$ with $K:=K\big(t_0,t_1,L,p, \sup_{n\in \mathbb{N}} E|x_{t_0}^n|^p, \sup_{n\in \mathbb{N}}EM_n^\frac{p}{2},\sup_{n\in \mathbb{N}}EM_n' \big)$ which does not depend on $n$.
\end{lemma}
\begin{proof}
The lemma follows immediately from Lemma \ref{lem:tem:conv:onestep} and Lemma \ref{lem:mb:tes}.
\end{proof}
\subsection{Convergence in $\mathcal{L}^q$}
For every $R>0 $, we consider $\mathscr F_{t_0}$-measurable random variables $C_R$  which satisfy,
\begin{align}
\lim_{R \rightarrow \infty}P (C_R > f(R)) =0,  \label{eq:cr}
\end{align}
for  a non-decreasing function $f: \mathbb{R_+} \rightarrow \mathbb{R_+}$. This notation for the family of random variables with the above property will be used throughout  this article.

\begin{assumptionA} \label{as:local:lips:rc}
For every $R>0$ and $t \in [t_0,t_1]$,
\begin{align}
 (x-\bar{x})(b_t(x)-b_t(\bar{x})) \vee |\sigma_t(x)-\sigma_t(\bar{x})|^2 \vee \int_{Z} |\gamma_t(x,z)-\gamma_t(\bar{x},z)|^2 \nu(dz)\leq  C_R |x-\bar{x}|^2  \notag
\end{align}
almost surely whenever $|x|,  |\bar{x}| \leq R$.
\end{assumptionA}
\begin{assumptionA} \label{as:local:bound:b:rc}
For every $R>0$ and $t \in [t_0,t_1]$,
\begin{align*}
\sup_{|x| \leq R}| b_t(x)| \leq {C}_R,
\end{align*}
almost surely.
\end{assumptionA}


\begin{assumptionB} \label{as:conv:coeff}
For every $R>0$ and $B(R):= \{\omega\in \Omega: C_R \leq f(R)\}$,
\begin{align*}
 \lim_{n \rightarrow \infty}E\int_{t_0}^{t_1}I_{B(R)}\sup_{|x| \leq R}\{|b^n_t(x)-b_t(x)|^2+|\sigma^n_t(x)-\sigma_t(x)|^2\} dt &=0
\\
 \lim_{n \rightarrow \infty}E \int_{t_0}^{t_1}I_{B(R)}\sup_{|x| \leq R}\int_{Z}|\gamma^n_t(x,z)-\gamma_t(x,z)|^2\nu(dz) dt &=0.
\end{align*}
\end{assumptionB}
\begin{assumptionB} \label{as:conv:initial}
For every $n \in \mathbb{N}$, the initial values of SDE \eqref{eq:sdewrc} and scheme \eqref{eq:em:sdewrc} satisfy $|x_{t_0}-x_{t_0}^n| \stackrel{P}{\rightarrow} 0$ as $n \to \infty$.
\end{assumptionB}

We introduce  families of stopping times that shall be used frequently in this report. For every $R>0$ and $n \in \mathbb{N}$, let
\begin{align} \label{eq:stopping:time}
\pi_R:=\inf\{t \geq t_0: |x_t| \geq R\}, \,\, \pi_{nR}:=\inf\{t \geq t_0:|x_t^n| \geq R\}, \,\,
\tau_{nR}:=\pi_R  \wedge \pi_{nR}
\end{align}
almost surely.
\begin{theorem} \label{thm:conv:rc}
Let Assumptions A-3 to  A-8 be satisfied. Also assume that B-1 to B-6 hold. Then,
\begin{equation}
\lim_{n \rightarrow \infty} E \sup_{t_0 \leq t \leq t_1}|x_t-x_t^n|^q =0 \notag
\end{equation}
for all  $q < p$.
\end{theorem}
\begin{proof}
Let $e^n_t:=x_t-x^n_t$ and define
\begin{equation}
\bar b^n_t:= b_t(x_t)-b^n_t(x^n_{\kappa(n,t)}) , \bar \sigma^n_t:=\sigma_t(x_t)-\sigma_t^n(x^n_{\kappa(n,t)}), \bar{\gamma}^n_{t}(z):=\gamma_t(x_{t},z)-\gamma_t^n(x^n_{\kappa(n,t)},z) \label{eq:coeff_bar}
\end{equation}
almost surely for any $t \in [t_0,t_1]$. In this simplified notation, $e_t^n$ can be written as
\begin{align}  \label{eq:con:en}
e^n_t  = e^n_{t_0} + \int_{t_0}^t \bar b_s^n ds+\int_{t_0}^t  \bar \sigma_s^n dw_s +\int_{t_0}^t \int_Z \bar \gamma_s^n(z) \tilde{N}(ds, dz)
\end{align}
almost surely for any $t \in [t_0, t_1]$. Further, by using the stopping times defined in equation \eqref{eq:stopping:time} and random variables defined in \eqref{eq:cr}, let us partition the sample space $\Omega$ into two parts $\Omega_1$ and $\Omega_2$ where
\begin{align*}
\Omega_1& =\{\omega \in \Omega:\pi_R \leq t_1 \,  \text{or} \,  \pi_{nR} \leq t_1 \, \text{or} \,  C_R >f(R)\   \}
\\
& =\{\omega \in \Omega: \pi_R \leq t_1\} \cup \{\omega \in \Omega: \pi_{nR} \leq t_1\} \cup \{\omega \in \Omega: C_R >f(R)\}
\\
\Omega_2& =\Omega \backslash \Omega_1=\{\omega \in \Omega: \pi_R > t_1\} \cap \{\omega \in \Omega: \pi_{nR} > t_1\} \cap B(R)
\end{align*}
where $ B(R):= \{\omega \in \Omega: C_R \leq f(R)\}$ as defined in Assumption B-5.  Also note that $I_\Omega=I_{\Omega_1 \cup \Omega_2} \leq I_{\Omega_1}+I_{\Omega_2}$. By using this fact, for any $q <p$, one could write the following,
\begin{align} \label{eq:D1+D2}
E\sup_{t_0 \leq t \leq t_1}|e_t^n|^q   =  E\sup_{t_0 \leq t \leq t_1}|e_t^n|^q  I_{\Omega_1}+  E\sup_{t_0 \leq t \leq t_1}|e_t^n|^q  I_{\Omega_2} =: D_1+D_2.
\end{align}
 By the application of H\"older's inequality, Lemma \ref{lem:mb:rc} and Lemma \ref{lem:mb:tes} one could write,
\begin{align} \label{eq:D1}
D_1& :=E\sup_{t_0 \leq t \leq t_1}|e_t^n|^q  I_{\Omega_1} \leq \Big(E\sup_{t_0 \leq t \leq t_1}|e_t^n|^{q\frac{p}{q}}\Big)^\frac{q}{p} \Big(EI_{\Omega_1}\Big)^\frac{p-q}{p} \notag
\\
 & \leq  K  \Big(\frac{E|x_{\pi_R}|^p}{R^p} + \frac{E|x^n_{\pi_{nR}}|^p}{R^p} + P( \{\omega \in \Omega: C_R >f(R)\}) \Big)^\frac{p-q}{p} \notag
\\
 & \leq  K  \Big(\frac{1}{R^p} + P( \{\omega \in \Omega: C_R >f(R)\}) \Big)^\frac{p-q}{p}
\end{align}
where the constant $K>0$ does not depend on $n$. Having obtained  estimates for $D_1$, we now proceed to obtain the estimates for $D_2$. For this, we recall equation \eqref{eq:con:en} and use It\^o formula to obtain the following,
\begin{align} \label{eq:en2}
|e_t^n|^2 & =  |e_{t_0}^n|^2+2 \int_{t_0}^{t}  e^n_s \bar b_s^n ds + 2\int_{t_0}^{t}  e^n_s \bar \sigma^{n}_s  dw_s  +\int_{t_0}^{t}|\bar{\sigma}^n_s|^2ds \notag
\\
&+  2\int_{t_0}^{t} \int_{Z}  e^n_{s} \bar\gamma_s^n(z)   \tilde N(ds,dz) +\int_{t_0}^{t} \int_{Z}|\bar \gamma_s^n(z)|^2 N(ds,dz)
\end{align}
almost surely for any $t \in [t_0,t_1]$. Also, to estimate  the second term of \eqref{eq:en2},  one uses the following splitting,
\begin{align} \label{eq:ebn:splitting}
e^n_s \bar b^n_s
=(x_s-x_{\kappa(n,s)}^n) & (b_s(x_s)-b_s(x_{\kappa(n,s)}^n))+ (x_s-x_{\kappa(n,s)}^n)  (b_s(x_{\kappa(n,s)}^n)-b_s^n(x_{\kappa(n,s)}^n)) \notag
\\
&+ (x_{\kappa(n,s)}^n-x_s^n)(b_s(x_s)-b_s(x_{\kappa(n,s)}^n))\notag
\\
& +(x_{\kappa(n,s)}^n-x_s^n)(b_s(x_{\kappa(n,s)}^n)-b_s^n(x_{\kappa(n,s)}^n))
\end{align}
almost surely for any $s \in [t_0,t_1]$. Notice that $D_2$ is non-zero only on  $\Omega_2$, thus one can henceforth restrict all the calculations in the estimation of $D_2$ on the interval $[t_0, t_1 \wedge \tau_{nR})$ which also means that $|x_t| \vee |x_t^n| < R$ for any $t \in [t_0, t_1 \wedge \tau_{nR})$.  As a consequence, on the application of Assumption A-7 and Cauchy-Schwarz inequality, one obtains
\begin{align*}
e^n_s \bar b^n_s & \leq C_R |x_s-x_{\kappa(n,s)}^n|^2 + |x_s-x_{\kappa(n,s)}^n|  |b_s(x_{\kappa(n,s)}^n)-b_s^n(x_{\kappa(n,s)}^n)|
\\
&+ |x_{\kappa(n,s)}^n-x_s^n||b_s(x_s)-b_s(x_{\kappa(n,s)}^n)|+|x_{\kappa(n,s)}^n-x_s^n||b_s(x_{\kappa(n,s)}^n)-b_s^n(x_{\kappa(n,s)}^n)|
\end{align*}
almost surely for any $s \in [t_0, t_1 \wedge \tau_{nR})$.   By using Assumption A-8, this can further be estimated as
\begin{align} \label{eq:ebn}
e^n_s \bar b^n_s & \leq (2C_R+1) |x_s-x_s^n|^2+(2C_R+\frac{3}{2})|x_s^n-x_{\kappa(n,s)}^n|^2 +  |b_s(x_{\kappa(n,s)}^n)-b_s^n(x_{\kappa(n,s)}^n)|^2 \notag
\\
& \qquad \qquad+ 2C_R|x_s^n-x_{\kappa(n,s)}^n|
\end{align}
almost surely for any $s \in [t_0, t_1 \wedge \tau_{nR})$.  Now, by using the definition of $\Omega_2$ and of $\tau_{nR}$ in equation \eqref{eq:stopping:time}, one has
\begin{align} \label{eq:D2:new}
D_2& :=E\sup_{t_0 \leq t \leq t_1}|e_t^n|^q  I_{\Omega_2} \leq E\sup_{t_0 \leq t \leq t_1}|e_{t \wedge \tau_{nR}}^n|^q  I_{B(R)}.
\end{align}
Thus using the estimate obtained in \eqref{eq:ebn}, one obtains
\begin{align}
E  \sup_{t_0 \leq t \leq u} & |e_{t \wedge \tau_{nR}}^n|^2  I_{B(R)}  \leq  E|e^n_{t_0}|^2+ E(2C_R+1)\int_{t_0}^{u \wedge \tau_{nR}}|e^n_s|^2 I_{B(R)} ds \notag
\\
& + E(2C_R+\frac{3}{2})\int_{t_0}^{u \wedge \tau_{nR}}|x^n_s-x^n_{\kappa(n,s)}|^2I_{B(R)}ds  \notag
\\
&  + 2E C_R  \int_{t_0}^{u \wedge \tau_{nR}} |x_s^n-x_{\kappa(n,s)}^n|I_{B(R)}ds \notag
\\
&+E\int_{t_0}^{u \wedge \tau_{nR}}|b_s(x^n_{\kappa(n,s)})-b^n_s(x^n_{\kappa(n,s)})|^2 I_{B(R)} ds  \notag
\\
& + 2E\sup_{t_0 \leq t \leq u} \Big|\int_{t_0}^{t \wedge \tau_{nR}} I_{B(R)} e^n_s \bar \sigma^{n}_s  dw_s \Big| +E\int_{t_0}^{u \wedge \tau_{nR}}|\bar{\sigma}^n_s|^2I_{B(R)}ds \notag
\\
& +  2E\sup_{t_0 \leq t \leq u}\Big|\int_{t_0}^{t \wedge \tau_{nR}} \int_{Z} I_{B(R)} e^n_{s} \bar\gamma_s^n(z)   \tilde N(ds,dz)\Big| \notag
\\
& +E \sup_{t_0 \leq t \leq u}\int_{t_0}^{t \wedge \tau_{nR}} \int_{Z} I_{B(R)}|\bar \gamma_s^n(z)|^2 N(ds,dz) \notag
\\
&=:F_1+F_2+F_3+F_4+F_5+F_6+F_7+F_8+F_9 \label{eq:F1+F9}
\end{align}
for every $R>0$ and $u \in [t_0, t_1 \wedge \tau_{nR})$. Here $F_1:=E|e^n_{t_0}|^2$.  $F_2$ is estimated easily by
\begin{align} \label{eq:F2}
F_2 & :=E(2C_R+1)\int_{t_0}^{u \wedge \tau_{nR}}|e^n_s|^2I_{B(R)}ds  \notag
\\
&\leq (2f(R)+1)\int_{t_0}^{u}E\sup_{t_0 \leq r \leq s}|e_{r \wedge \tau_{nR}}^n|^2 I_{B(R)}ds
\end{align}
for every $R>0$ and $u \in [t_0, t_1 \wedge \tau_{nR})$.
Further,
\begin{align} \label{eq:F3}
F_3 & :=E(2C_R+\frac{3}{2})\int_{t_0}^{u \wedge \tau_{nR}}|x^n_s-x^n_{\kappa(n,s)}|^2I_{B(R)}ds  \notag
\\
& \leq (f(R)+1)K\sup_{t_0 \leq t \leq t_1} E|x^n_t-x^n_{\kappa(n,t)}|^2
\end{align}
and similarly,  term $F_4$ can be estimated by
\begin{align} \label{eq:F4}
F_4&:=2E C_R  \int_{t_0}^{u \wedge \tau_{nR}} |x_s^n-x_{\kappa(n,s)}^n|I_{B(R)} ds \leq f(R) K \sup_{t_0 \leq t \leq t_1} E|x_t^n-x_{\kappa(n,t)}^n|
\end{align}
for every $R>0$. Again,  term $F_5$ has following estimate,
\begin{align} \label{eq:F5}
F_5& := E\int_{t_0}^{u \wedge \tau_{nR}}|b_s(x^n_{\kappa(n,s)})-b^n_s(x^n_{\kappa(n,s)})|^2 I_{B(R)} ds  \notag
\\
&\leq  E\int_{t_0}^{t_1}I_{\{t_0 \leq s <\tau_{nR}\}} I_{B(R)}|b_s(x^n_{\kappa(n,s)})-b^n_s(x^n_{\kappa(n,s)})|^2 ds.
\end{align}
To estimate the term $F_6$, one uses Burkholder-Davis-Gundy inequality to write
\begin{align*}
F_6 & :=2E\sup_{t_0 \leq t \leq u}\Big|\int_{t_0}^{t \wedge \tau_{nR}} I_{B(R)} e^n_s \bar \sigma^{n}_s  dw_s \Big|  \leq K E\Big(\int_{t_0}^{u \wedge \tau_{nR}}  I_{B(R)}|e^n_s|^2 |\bar \sigma^{n}_s|^2  ds \Big)^\frac{1}{2}
\\
& \leq K E\sup_{t_0 \leq s \leq u}|e^n_{s\wedge \tau_{nR}}|I_{B(R)} \Big(\int_{t_0}^{u \wedge \tau_{nR} }   I_{B(R)} |\bar \sigma^{n}_s|^2  ds \Big)^\frac{1}{2}
\end{align*}
which on the application of Young's inequality gives
\begin{align} \label{eq:F6+F7:a}
F_6+F_7 \leq  \frac{1}{8} E\sup_{t_0 \leq s \leq u}|e^n_{s\wedge \tau_{nR}}|^2 I_{B(R)} + K E\int_{t_0}^{u \wedge \tau_{nR}}   I_{B(R)} |\bar \sigma^{n}_s|^2  ds
\end{align}
for any $R>0$  and $u \in [t_0, t_1]$ where constant $K>0$ does not depend on $R$ and $n$. In order to estimate the second term of the above inequality, one uses the following splitting of $\bar{\sigma}_s^n$,
\begin{align} \label{eq:sigma:splitting}
\bar{\sigma}^n_s &  = (\sigma_s(x_s)-\sigma_s(x_s^n))+(\sigma_s(x_s^n)-\sigma_s(x^n_{\kappa(n,s)}))+(\sigma_s(x^n_{\kappa(n,s)})-\sigma_s^n(x^n_{\kappa(n,s)}))
\end{align}
almost surely for any $s \in [t_0, t_2]$. As before, one again notices that $|x_s|\leq R$ and $|x_s^n|\leq R$ whenever $s \in [t_0, t_1 \wedge \tau_{nR})$. Thus on the application of Assumption A-7, one obtains
\begin{align*}
|\bar{\sigma}^n_s|^2 \leq 3C_R|e_s^n|^2+3C_R|x_s^n-x^n_{\kappa(n,s)}|^2+3|\sigma_s(x^n_{\kappa(n,s)})-\sigma_s^n(x^n_{\kappa(n,s)})|^2
\end{align*}
almost surely $s \in [t_0, t_1 \wedge \tau_{nR})$. Hence substituting this estimate in inequality \eqref{eq:F6+F7:a} gives
\begin{align} \label{eq:F6+F7}
F_6+F_7 & \leq  \frac{1}{8} E\sup_{t_0 \leq s \leq u}|e^n_{s\wedge \tau_{nR}}|^2 I_{B(R)} + K  f(R) \int_{t_0}^{u}  E\sup_{t_0 \leq r \leq s}|e^n_{r\wedge \tau_{nR}}|^2 I_{B(R)}  ds \notag
\\
&+ K  f(R) \sup_{t_0 \leq s \leq t_1} E |x_s^n-x^n_{\kappa(n,s)}|^2  \notag
\\
& + K E\int_{t_0}^{t_1}  I_{\{t_0 \leq s < \tau_{nR}\}} I_{B(R)}  |\sigma_s(x^n_{\kappa(n,s)})-\sigma_s^n(x^n_{\kappa(n,s)})|^2 ds
\end{align}
for any $u \in [t_0, t_1]$. Further, one proceeds as above in the similar way to the derivation of \eqref{eq:F6+F7:a} and uses Lemma \ref{lem:maximal:inequality} to obtain
\begin{align} \label{eq:F8+F9:o}
F_8+F_9 \leq & \frac{1}{8} E\sup_{t_0 \leq s \leq u} |e^n_{s  \wedge \tau_{nR}}|^2 I_{B(R)} \notag
\\
& + K  E\int_{t_0}^{u \wedge \tau_{nR} }\int_{Z}  I_{B(R)}  |\bar\gamma_s^n(z)|^2  \nu(dz) ds
\end{align}
for any $u \in [t_0, t_1]$. In order to estimate the second term of the above inequality, one uses the following splitting,
\begin{align} \label{eq:gamma:splitting}
\hspace{-1mm}\bar{\gamma}^n_s(z) = & (\gamma(x_s,z)-\gamma_s(x_s^n,z))+(\gamma_s(x_s^n,z)-\gamma_s(x^n_{\kappa(n,s)},z) \notag
\\
&+(\gamma_s(x^n_{\kappa(n,s)},z)-\gamma_s^n(x^n_{\kappa(n,s)},z))
\end{align}
almost surely for any $s \in [t_0,t_1]$.  Thus, by using the Assumption  A-7, one has
\begin{align} \label{eq:F8+F9}
F_8+F_9  \leq & \frac{1}{8} E\sup_{t_0 \leq s \leq u} |e^n_{s  \wedge \tau_{nR}}|^2 I_{B(R)} + K f(R) E \int_{t_0}^{u } E\sup_{t_0 \leq r \leq s} |e^n_{r  \wedge \tau_{nR}}|^2 I_{B(R)}  ds \notag
\\
&  + K  f(R) \sup_{t_0\leq s \leq t_1}E|x_s^n-x^n_{\kappa(n,s)}|^2  \notag
\\
& + K  E \int_{t_0}^{t_1 } \int_Z  I_{\{t_0 \leq s <  \tau_{nR}\}} I_{B(R)} |\gamma_s(x^n_{\kappa(n,s)},z)-\gamma_s^n(x^n_{\kappa(n,s)},z)|^2  \nu(dz)   ds
\end{align}
for any $u \in [t_0, t_1]$. On combining  estimates obtained in \eqref{eq:F2}, \eqref{eq:F3}, \eqref{eq:F4}, \eqref{eq:F6+F7} and \eqref{eq:F8+F9} in \eqref{eq:F1+F9} and then applying Gronwall's inequality, one obtains
\begin{align}
E  &\sup_{t_0 \leq t \leq t_1}|e_{t \wedge \tau_{nR}}^n|^2 I_{B(R)} \leq  \exp(Kf(R)) \Big\{E|e_{t_0}^n|^2 + Kf(R)\sup_{t_0\leq s \leq t_1}E|x^n_{s}-x^n_{\kappa(n,s)}|^2 \notag
\\
& \qquad \qquad + Kf(R)\big(\sup_{t_0\leq s \leq t_1}E|x^n_{s}-x^n_{\kappa(n,s)}|^2\big)^\frac{1}{2} \notag
\\
&+K E\int_{t_0}^{t_1}I_{\{t_0 \leq s <  \tau_{nR}\}}I_{B(R)}|b_s(x^n_{\kappa(n,s)})-b^n_s(x^n_{\kappa(n,s)})|^2 ds \notag
\\
&+ K E\int_{t_0}^{t_1}I_{\{t_0 \leq s <  \tau_{nR}\}}I_{B(R)}|\sigma_s(x^n_{\kappa(n,s)})-\sigma^n_s(x^n_{\kappa(n,s)})|^2 ds \notag
\\
&  + K E\int_{t_0}^{t_1} \int_Z I_{\{t_0 \leq s < \tau_{nR}\}}I_{B(R)}|\gamma_s(x^n_{\kappa(n,s)}, z)-\gamma^n_s(x^n_{\kappa(n,s)},z)|^2 \nu(dz)ds  \Big\}.\notag
\end{align}
Hence, by the application of Lemma \ref{lem:tem:new:conv:onestep},  Assumptions B-5 and B-6, one obtains
\begin{align*}
E  \sup_{t_0 \leq t \leq t_1}|e_{t \wedge \tau_{nR}}^n|^2 I_{B(R)} \to 0 \mbox{\,\,as\,\,} n \to \infty
\end{align*}
for every $R>0$. Consequently $\sup_{t_0 \leq t \leq t_1}|e_{t \wedge \tau_{nR}}^n|I_{B(R)}   \to 0$ in probability, as $n\to \infty$.
By Lemma \ref{lem:mb:rc} and  Lemma \ref{lem:mb:tes}, we have that the sequence of random variables,  $(\sup_{t_0 \leq t \leq t_1}|e_{t \wedge \tau_{nR}}^n|^q I_{B(R)} )_{n \in \mathbb{N}}$ is  uniformly integrable for any $q<p$. Hence, for each $R >0 $ we have
$$
 E \sup_{t_0 \leq t \leq t_1}|e_{t \wedge \tau_{nR}}^n|^q I_{B(R)}  \to 0, \text{as $n\to \infty$}
$$
which implies from inequality \eqref{eq:D2:new} that $D_2 \to 0$ as $n\to \infty$ for every $R>0$.
Also by choosing sufficiently large $R>0$ in inequality \eqref{eq:D1} along with equation \eqref{eq:cr}, one obtains $D_1 \to 0$. This complete the proof.
\end{proof}
\subsection{Rate of Convergence}
In order to obtain rate of convergence of the scheme \eqref{eq:em:sdewrc}, one replaces Assumption A-7 by the following assumptions.
\begin{assumptionA} \label{as:lipschitz:rc}
There exist constants $C>0$, $q \geq 2$ and $\chi>0$ such that
\begin{align}
 (x-\bar{x})( b_t(x)-b_t(\bar{x})) \vee |\sigma_t(x)-\sigma_t(\bar{x})|^2 &\vee \int_{Z} |\gamma_t(x,z)-\gamma_t(\bar{x},z)|^2 \nu(dz)  \leq  C |x-\bar{x}|^2  \notag
 \\
  \int_{Z} |\gamma_t(x,z)-\gamma_t(\bar{x},z)|^q \nu(dz) & \leq C |x-\bar{x}|^q \notag
 \\
 |b_t(x)-b_t(\bar{x})|^2 & \leq C(1+|x|^\chi+|\bar{x}|^\chi)|x-\bar{x}|^2 \label{eq:b}
\end{align}
almost surely for any $t \in [t_0,t_1]$, $x,\bar{x} \in \mathbb{R}^d$  and a $\delta \in (0,1)$ such that $\max\big\{(\chi+2)q,\frac{q \chi}{2}\frac{q+\delta}{\delta}\big\}\leq p$.
\end{assumptionA}
\begin{remark}\label{rem:poly:b:rc}
Due to \eqref{eq:b} and Assumption A-8, one immediately obtains
\begin{align*}
|b_t(x)|^2 \leq K(1+|x|^{\chi+2})
\end{align*}
almost surely for any $t \in [t_0, t_1]$ and $x \in \mathbb{R}^d$.
\end{remark}

Furthermore, one replaces Assumption B-5 by the following assumption.
\begin{assumptionB} \label{as:coeff:conv:rate:rc}
There exists a constant $C>0$ such that
\begin{align}
 E\int_{t_0}^{t_1}\{|b^n_t(x_{\kappa(n,t)}^n)-b_t(x_{\kappa(n,t)}^n)|^q +|\sigma^n_t(x_{\kappa(n,t)}^n)-\sigma_t(x_{\kappa(n,t)}^n)|^q\}dt &\leq C n^{-\frac{q}{q+\delta}} \notag
\\
 E \int_{t_0}^{t_1}\Big(\int_{Z}|\gamma^n_t(x_{\kappa(n,t)}^n,z)-\gamma_t(x_{\kappa(n,t)}^n,z)|^{\zeta}\nu(dz)\Big)^{\frac{q}{\zeta}} dt &\leq C n^{-\frac{q}{q+\delta}} \notag
\end{align}
for $\zeta=2,q$.
\end{assumptionB}

Finally, Assumption B-6  is replaced by the following assumption.
\begin{assumptionB} \label{as:initial:rate:rc}
There exists a constant $C>0$ such that
\begin{align}
E|x_{t_0}-x_{t_0}^n|^q \leq C n^{-\frac{q}{q+\delta}}. \notag
\end{align}
\end{assumptionB}

\begin{theorem} \label{thm:rate:rc:em}
Let Assumptions A-3 to  A-6, A-8 and A-9 be satisfied. Also suppose that Assumptions B-1 to B-4,  B-7 and B-8 hold. Then
\begin{align}
E\sup_{t_0 \leq t \leq t_1}|x_{t}-x_{t}^n|^q \leq K n^{-\frac{q}{q+\delta}} \notag
\end{align}
where constant $K>0$ does not depend on $n$.
\end{theorem}
\begin{proof}
First of all, let us recall the notations used in the proof of Theorem \ref{thm:conv:rc}. By the application of It\^o formula, one obtains
\begin{align}
|e_t^n|^q &= |e_{t_0}^n|^q+ q \int_{t_0}^{t} |e_s^n|^{q-2} e_s^n \bar b_s^n ds + q\int_{t_0}^{t} |e_s^n|^{q-2} e_s^n \bar \sigma_s^n dw_s  \notag
\\
& + \frac{q(q-2)}{2} \int_{t_0}^{t} |e_s^n|^{q-4}|\bar \sigma_s^{n*} e_s^n|^2ds +\frac{q}{2}\int_{t_0}^{t} |e_s^n|^{q-2}|\bar \sigma_s^n|^2 ds  \notag
\\
&+  q\int_{t_0}^{t} \int_{Z} |e_s^n|^{q-2} e_s^n \bar \gamma_s^n(z)    \tilde N(ds,dz) \notag
\\
+\int_{t_0}^{t} &\int_{Z}\{ |e_s^n+\bar \gamma_s^n(z)|^q-|e_s^n|^q-q|e_s^n|^{q-2} e_s^n \bar \gamma_s^n(z) \}N(ds,dz) \label{eq:en}
\end{align}
almost surely for any $t \in [t_0,t_1]$. In Theorem \ref{thm:conv:rc}, the splitting given in \eqref{eq:ebn:splitting} is used to prove the $\mathcal{L}^q$ convergence of the scheme \eqref{eq:em:sdewrc}. In order to obtain a rate of convergence of scheme \eqref{eq:em:sdewrc},  one uses the following splitting,
\begin{align} \label{eq:ebn:splitting:rate}
e^n_s \bar b^n_s =& (x_s-x_s^n)(b_s(x_s)-b_s(x^n_s)) + (x_s-x^n_s)(b_s(x^n_s)-b_s(x^n_{\kappa(n,s)})) \notag
\\
& +(x_s-x^n_s)(b_s(x^n_{\kappa(n,s)})-b_s^n(x^n_{\kappa(n,s)}))
\end{align}
which on the application of Assumption A-9, Cauchy-Schwarz inequality and Young's inequality gives
\begin{align} \label{eq:en2:rate}
|e^n_s|^{q-2}e^n_s \bar b^n_s & \leq  K|e^n_s|^q + K|b_s(x^n_s)-b_s(x^n_{\kappa(n,s)})|^q  + K|b_s(x^n_{\kappa(n,s)})-b_s^n(x^n_{\kappa(n,s)})|^q
\end{align}
almost surely for any $s \in[t_0,t_1]$.  Therefore by taking suprema over $[t_0, u]$ for any $u \in [t_0,t_1]$ and expectations, one has
\begin{align} \label{eq:G1+G9}
E\sup_{t_0 \leq  t \leq u}|e_t^n|^q  \leq & E|e_{t_0}^n|^q+ K E\int_{t_0}^{u} |e_s^n|^q ds + K E\int_{t_0}^{u}  |b_s(x^n_s)-b_s(x^n_{\kappa(n,s)})|^q ds \notag
\\
& \hspace{-2cm} + K E\int_{t_0}^{u} |b_s(x^n_{\kappa(n,s)})-b_s^n(x^n_{\kappa(n,s)})|^q ds  + qE\sup_{t_0 \leq t \leq u}\Big|\int_{t_0}^{t} |e_s^n|^{q-2} e_s^n \bar \sigma_s^n dw_s\Big| \notag
\\
& + \frac{q(q-2)}{2} E\int_{t_0}^{u} |e_s^n|^{q-4}|\bar \sigma_s^{n*} e_s^n|^2ds +\frac{q}{2}E\int_{t_0}^{u} |e_s^n|^{q-2}|\bar \sigma_s^n|^2 ds  \notag
\\
&+  qE\sup_{t_0 \leq t \leq u}\Big|\int_{t_0}^{t} \int_{Z} |e_s^n|^{q-2} e_s^n \bar \gamma_s^n(z)    \tilde N(ds,dz)\Big| \notag
\\
&+E\sup_{t_0 \leq t \leq u}\int_{t_0}^{t} \int_{Z}\{ |e_s^n|^{q-2}|\bar \gamma_s^n(z)|^2+|\bar \gamma_s^n(z)|^q \}N(ds,dz) \notag
\\
=& G_1+G_2+G_3+G_4+G_5+G_6+G_7+G_8+G_9
\end{align}
for any $u \in [t_0, t_1]$. Here $G_1:=E|e_{t_0}^n|^q$ and $G_2$ can be estimated by
\begin{align} \label{eq:G2}
G_2:= K E\int_{t_0}^{u} |e_s^n|^q ds \leq K \int_{t_0}^{u} E\sup_{t_0 \leq r \leq s}|e_r^n|^q ds
\end{align}
for any $u \in [t_0, t_1]$. By the application of Assumption A-9, H\"older's inequality  and Lemma \ref{lem:mb:tes}, $G_3$ can be estimated by
\begin{align} \label{eq:G3}
G_3 :=K E\int_{t_0}^{u}  &|b_s(x^n_s)-b_s(x^n_{\kappa(n,s)})|^q ds \notag
\\
&\hspace{-2cm}\leq K \int_{t_0}^{u} \left( 1+E|x^n_s|^{\chi\frac{q}{2}\frac{q+\delta}{\delta}}+E|x^n_{\kappa(n,s)}|^{\chi\frac{q}{2}\frac{q+\delta}{\delta}}\right)^\frac{\delta}{q+\delta} \left( E|x^n_s-x^n_{\kappa(n,s)}|^{q+\delta}\right)^\frac{q}{q+\delta} ds \notag
\\
& \leq K \int_{t_0}^{t_1}  \left( E|x^n_s-x^n_{\kappa(n,s)}|^{q+\delta}\right)^\frac{q}{q+\delta} ds.
\end{align}
Further, $G_4$ can be estimated by
\begin{align} \label{eq:G4}
G_4&:=K E\int_{t_0}^{u} |b_s(x^n_{\kappa(n,s)})-b_s^n(x^n_{\kappa(n,s)})|^q ds \notag
\\
& \leq K E\int_{t_0}^{t_1} |b_s(x^n_{\kappa(n,s)})-b_s^n(x^n_{\kappa(n,s)})|^q ds.
\end{align}
By  the application of Burkholder-Davis-Gundy inequality, one obtains
\begin{align*}
G_5 & := qE\sup_{t_0 \leq t \leq u}\Big|\int_{t_0}^{t} |e_s^n|^{q-2} e_s^n \bar \sigma_s^n dw_s\Big| \leq K E\Big(\int_{t_0}^{u} |e_s^n|^{2q-2}|\bar \sigma_s^n|^2 ds\Big)^\frac{1}{2}
\\
& \leq K E\sup_{t_0 \leq s \leq u}|e_s^n|^{q-1}\Big(\int_{t_0}^{u} |\bar \sigma_s^n|^2 ds\Big)^\frac{1}{2}
\end{align*}
which due to Young's inequality  and H\"older's inequality gives
\begin{align} \label{eq:G5:old}
G_5 \leq \frac{1}{8} E\sup_{t_0 \leq s \leq u}|e_s^n|^{q} + K E\int_{t_0}^{u} |\bar \sigma_s^n|^q ds
\end{align}
for any $u \in [t_0, t_1]$. Further, due to Cauchy-Schwarz inequality and Young's inequality, $G_6$ and $G_7$ can be estimated together by
\begin{align} \label{eq:G6+G7}
G_6+G_7&:=\frac{q(q-2)}{2} E\int_{t_0}^{u} |e_s^n|^{q-4}|\bar \sigma_s^{n*} e_s^n|^2ds +\frac{q}{2}E\int_{t_0}^{u} |e_s^n|^{q-2}|\bar \sigma_s^n|^2 ds \notag
\\
& \leq K  E\int_{t_0}^{u} |e_s^n|^{q-2}|\bar \sigma_s^n|^2 ds \leq K \int_{t_0}^{u} E\sup_{t_0 \leq r \leq s}|e_r^n|^{q} ds+ K E\int_{t_0}^{u}|\bar \sigma_s^n|^q ds
\end{align}
for any $u \in [t_0, t_1]$. On combining the estimated from \eqref{eq:G5:old} and \eqref{eq:G6+G7}, one has
\begin{align} \label{eq:G5+G6+G7:old}
G_5+G_6+G_7  \leq \frac{1}{8} E\sup_{t_0 \leq s \leq u}|e_s^n|^{q} + K \int_{t_0}^{u} E\sup_{t_0 \leq r \leq s}|e_r^n|^{q} ds+ K E\int_{t_0}^{u}|\bar \sigma_s^n|^q ds
\end{align}
for any $u \in [t_0, t_1]$. Now, one uses the splitting of $\bar \sigma_s^n$ given in \eqref{eq:sigma:splitting} along with Assumption A-9 to write
\begin{align} \label{eq:G5+G6+G7}
G_5+G_6+G_7  & \leq \frac{1}{8} E\sup_{t_0 \leq s \leq u}|e_s^n|^{q} + K \int_{t_0}^{u} E\sup_{t_0 \leq r \leq s}|e_r^n|^{q} ds+ K \int_{t_0}^{t_1}E|x_s^n-x_{\kappa(n,s)}^n|^q ds \notag
\\
&+ K E\int_{t_0}^{t_1}|\sigma_s(x_{\kappa(n,s)}^n)-\sigma_s^n(x_{\kappa(n,s)}^n)|^q ds
\end{align}
for any $u \in [t_0, t_1]$. Further, for estimating $G_8$, one uses the splitting of $\bar\gamma_s^n(z)$ given in \eqref{eq:gamma:splitting} to write
\begin{align}
G_8 & \leq E \sup_{t_0 \leq t \leq u} \Big| \int_{t_0}^{t} \int_{Z} |e_s^n|^{q-2} e_s^n  \{\gamma_s(x_s,z)-\gamma_s(x_s^n,z)\}    \tilde N(ds,dz) \Big| \notag
\\
& +E \sup_{t_0 \leq t \leq u} \Big| \int_{t_0}^{t} \int_{Z} |e_s^n|^{q-2} e_s^n \{ \gamma_s(x_s^n,z)-\gamma_s(x_{\kappa(n,s)}^n,z)   \} \tilde N(ds,dz) \Big| \notag
\\
& + E \sup_{t_0 \leq t \leq u} \Big| \int_{t_0}^{t} \int_{Z} |e_s^n|^{q-2} e_s^n \{\gamma_s(x_{\kappa(n,s)}^n,z)-\gamma_s^n(x_{\kappa(n,s)}^n,z) \}    \tilde N(ds,dz) \Big| \notag
\end{align}
which due to Lemma \ref{lem:maximal:inequality} gives
\begin{align}
G_8 & \leq E  \Big( \int_{t_0}^{u} \int_{Z} |e_s^n|^{2q-2}   |\gamma_s(x_s,z)-\gamma_s(x_s^n,z)|^2  \nu(dz)ds \Big)^\frac{1}{2} \notag
\\
& +E \Big( \int_{t_0}^{u} \int_{Z} |e_s^n|^{2q-2} |\gamma_s(x_s^n,z)-\gamma_s(x_{\kappa(n,s)}^n,z) |^2    \nu(dz)ds \Big)^\frac{1}{2} \notag
\\
& + E  \Big( \int_{t_0}^{u} \int_{Z} |e_s^n|^{2q-2}  |\gamma_s(x_{\kappa(n,s)}^n,z)-\gamma_s^n(x_{\kappa(n,s)}^n,z)|^2    \nu(dz)ds \Big)^\frac{1}{2} \notag
\end{align}
for any $u \in [t_0,t_1]$. Then on the application of Young's inequality and H\"older's inequality, one obtains,
\begin{align}
G_8 & \leq \frac{1}{8} E \sup_{t_0 \leq s \leq u} |e_s^n|^q + E \int_{t_0}^{u} \Big(\int_{Z}    |\gamma_s(x_s,z)-\gamma_s(x_s^n,z)|^2  \nu(dz)\Big)^\frac{q}{2} ds \notag
\\
& + E  \int_{t_0}^{u} \Big(\int_{Z}  |\gamma_s(x_s^n,z)-\gamma_s(x_{\kappa(n,s)}^n,z) |^2    \nu(dz) \Big)^\frac{q}{2} ds\notag
\\
& + E\int_{t_0}^{u}\Big(  \int_{Z}   |\gamma_s(x_{\kappa(n,s)}^n,z)-\gamma_s^n(x_{\kappa(n,s)}^n,z)|^2    \nu(dz) \Big)^\frac{q}{2}ds. \notag
\end{align}
Thus by using Assumption A-9, one has
\begin{align} \label{eq:G8}
G_8 & \leq \frac{1}{8} E \sup_{t_0 \leq s \leq u} |e_s^n|^q +  \int_{t_0}^{u}  E\sup_{t_0 \leq r \leq s}|e_s^n|^q  ds  +   \int_{t_0}^{t_1} E|x_s^n-x_{\kappa(n,s)}^n|^q   ds\notag
\\
& + E\int_{t_0}^{t_1}\Big(  \int_{Z}   |\gamma_s(x_{\kappa(n,s)}^n,z)-\gamma_s^n(x_{\kappa(n,s)}^n,z)|^2    \nu(dz) \Big)^\frac{q}{2}ds
\end{align}
for any $u \in [t_0, t_1]$. Finally, one could write $G_9$ as
\begin{align} \label{eq:G9:old}
G_9 & := E\sup_{t_0 \leq t \leq u}\int_{t_0}^{t} \int_{Z}\{ |e_s^n|^{q-2}|\bar \gamma_s^n(z)|^2+|\bar \gamma_s^n(z)|^q \}N(ds,dz)  \notag
\\
&= E\int_{t_0}^{u} \int_{Z} |e_s^n|^{q-2}|\bar \gamma_s^n(z)|^2 \nu(dz) ds +E\int_{t_0}^{u} \int_{Z} |\bar \gamma_s^n(z)|^q \nu(dz) ds =:H_1+H_2
\end{align}
for any $u \in [t_0,t_1]$. In order to estimate the first term $H_1$ on the right hand side of the inequality \eqref{eq:G9:old} along with Assumption A-9, one recalls the splitting of $\gamma_s^n(z)$ given in \eqref{eq:gamma:splitting}  to get the following estimate,
\begin{align}
H_1  \leq & K E\int_{t_0}^{u} |e_s^n|^{q} ds + K E\int_{t_0}^{u}  |e_s^n|^{q-2}|x_s^n-x_{\kappa(n,s)}^n|^2 ds  \notag
\\
&+ E\int_{t_0}^{u} \int_{Z} |e_s^n|^{q-2}|\gamma_s(x_{\kappa(n,s)}^n,z)-\gamma_s^n(x_{\kappa(n,s)}^n,z) |^2 \nu(dz)ds \notag
\end{align}
for any $u \in [t_0, t_1]$. By the application of Young's inequality, one obtains
\begin{align} \label{eq:H1}
H_1 & \leq  K \int_{t_0}^{u} E\sup_{t_0 \leq r \leq s} |e_r^n|^{q} ds + K\int_{t_0}^{t_1}  E|x_s^n-x_{\kappa(n,s)}^n|^q ds  \notag
\\
&+ K E\int_{t_0}^{t_1} \Big(\int_{Z} |\gamma_s(x_{\kappa(n,s)}^n,z)-\gamma_s^n(x_{\kappa(n,s)}^n,z) |^2 \nu(dz)\Big)^\frac{q}{2} ds
\end{align}
 for any $u \in [t_0, t_1]$. For the second term $H_2$ on the right hand side of the inequality \eqref{eq:G9:old} along with Assumption A-9, one again uses the splitting of $\bar{\gamma}_s^n(z)$ given in equation \eqref{eq:gamma:splitting} to get the following estimate,
\begin{align} \label{eq:H2}
H_2 & \leq  K \int_{t_0}^{u} E\sup_{t_0 \leq r \leq s} |e_r^n|^{q} ds + K\int_{t_0}^{t_1}  E|x_s^n-x_{\kappa(n,s)}^n|^q ds  \notag
\\
&+ K E\int_{t_0}^{t_1} \int_{Z} |\gamma_s(x_{\kappa(n,s)}^n,z)-\gamma_s^n(x_{\kappa(n,s)}^n,z) |^q \nu(dz) ds
\end{align}
for any $u \in [t_0, t_1]$.  Hence on combining the estimates obtained in \eqref{eq:H1} and \eqref{eq:H2} in \eqref{eq:G9:old}, one obtains
\begin{align} \label{eq:G9}
G_9 & \leq  K \int_{t_0}^{u} E\sup_{t_0 \leq r \leq s} |e_r^n|^{q} ds + K\int_{t_0}^{t_1}  E|x_s^n-x_{\kappa(n,s)}^n|^q ds  \notag
\\
& + K E\int_{t_0}^{t_1} \Big(\int_{Z} |\gamma_s(x_{\kappa(n,s)}^n,z)-\gamma_s^n(x_{\kappa(n,s)}^n,z) |^2 \nu(dz)\Big)^\frac{q}{2} ds \notag
\\
&+ K E\int_{t_0}^{t_1} \int_{Z} |\gamma_s(x_{\kappa(n,s)}^n,z)-\gamma_s^n(x_{\kappa(n,s)}^n,z) |^q \nu(dz) ds
\end{align}
 for any $u \in [t_0, t_1]$.

Thus one can substitute   estimates from \eqref{eq:G2}, \eqref{eq:G3}, \eqref{eq:G4}, \eqref{eq:G5+G6+G7}, \eqref{eq:G8} and \eqref{eq:G9} in \eqref{eq:G1+G9} and then apply Gronwall's inequality to obtain
\begin{align*}
E\sup_{t_0 \leq  t \leq t_1}&|e_t^n|^q  \leq  E|e_{t_0}^n|^q + K \int_{t_0}^{t_1}  \left( E|x^n_s-x^n_{\kappa(n,s)}|^{q+\delta}\right)^\frac{q}{q+\delta}ds
\\
& + K\int_{t_0}^{t_1}  E|x_s^n-x_{\kappa(n,s)}^n|^q ds + K E\int_{t_0}^{t_1} |b_s(x^n_{\kappa(n,s)})-b_s^n(x^n_{\kappa(n,s)})|^q ds \notag
\\
&   + K E\int_{t_0}^{t_1} |\sigma_s(x^n_{\kappa(n,s)})-\sigma_s^n(x^n_{\kappa(n,s)})|^q ds
\\
& + K E\int_{t_0}^{t_1} \Big(\int_{Z} |\gamma_s(x_{\kappa(n,s)}^n,z)-\gamma_s^n(x_{\kappa(n,s)}^n,z) |^2 \nu(dz)\Big)^\frac{q}{2} ds \notag
\\
&+ K E\int_{t_0}^{t_1} \int_{Z} |\gamma_s(x_{\kappa(n,s)}^n,z)-\gamma_s^n(x_{\kappa(n,s)}^n,z) |^q \nu(dz) ds.
\end{align*}
By the application of Assumptions  B-7, B-8 and Lemma \ref{lem:tem:new:conv:onestep}, one obtains,
\begin{align*}
E\sup_{t_0 \leq  t \leq t_1}&|e_t^n|^q \leq Kn^{-\frac{q}{q+\delta}}
\end{align*}
which completes the proof.
\end{proof}

\subsection{A Simple Example}
We now introduce a tamed Euler scheme of SDEs driven by L\'evy noise which have coefficients that are not random. For this purpose, we only highlight the modifications needed in the settings of our previous discussion. In SDE \eqref{eq:sdewrc},  $b_t(x)$ and $\sigma_t(x)$ are  $\mathscr{B}([0,T]) \otimes \mathscr{B}(\mathbb R^d)$-measurable functions with values in $\mathbb{R}^d$ and $\mathbb{R}^{d\times m}$ respectively. Also $\gamma_t(x,z)$ is a $\mathscr{B}([0,T]) \otimes \mathscr{B}(\mathbb R^d)\otimes \mathscr{Z}$-measurable  function with values in $\mathbb{R}^d$. Moreover, one modifies Assumptions A-5 and A-6 by assigning $M=M'=1$. Further, for every $n \in \mathbb{N}$, the scheme \eqref{eq:em:sdewrc} is given by defining
\begin{align} \label{eq:sde:coeff}
b^n_t(x)=\frac{b_t(x)}{1+n^{-\theta}|b_t(x)|}, \sigma_t^n(x)=\sigma_t(x) \mbox{ and } \gamma_t^n(x,z)=\gamma_t(x,z)
\end{align}
with $\theta \in (0,\frac{1}{2}]$ for any $t \in [t_0,t_1]$, $x \in \mathbb{R}^d$ and $z \in Z$. Then, it is easy to observe that Assumptions B-2 to B-4  hold since $M_n=M_n'=1$ and $\theta \in (0,\frac{1}{2}]$. Hence Lemmas [\ref{lem:mb:rc}, \ref{lem:tem:conv:onestep}, \ref{lem:mb:tes}, \ref{lem:tem:new:conv:onestep}] follow immediately. Finally, $\mathscr{F}_{t_0}$ measurable random variable $C_R$ in Assumptions A-7 and A-8 is a constant for every $R$.  In this new settings, one obtains the following corollaries for SDE \eqref{eq:sdewrc} and scheme \eqref{eq:em:sdewrc} with coefficients given by \eqref{eq:sde:coeff}.
\begin{corollary}
Let Assumptions A-3 to  A-8 be satisfied by the coefficients of SDE given immediately above. Also assume that B-1 and  B-6 hold. Then,  the numerical scheme \eqref{eq:em:sdewrc} with coefficients given by \eqref{eq:sde:coeff} converges to the solution of SDE \eqref{eq:sdewrc} in $\mathcal{L}^q$ sense i.e.
\begin{equation}
\lim_{n \rightarrow \infty} E \sup_{t_0 \leq t \leq t_1}|x_t-x_t^n|^q =0 \notag
\end{equation}
for all  $q < p$.
\end{corollary}
\begin{proof}
Assumption A-7 and A-8 are satisfied on taking $f(R)=C_R$ in equation \eqref{eq:cr}. For Assumption B-5, one observes due to \eqref{eq:sde:coeff} and Assumption A-8,
\begin{align}
 E\int_{t_0}^{t_1} I_{B(R)} \sup_{|x| \leq R}|b_t^n(x)-b_t(x)|^2dt& \leq n^{-2\theta} E\int_{t_0}^{t_1} I_{B(R)} \sup_{|x| \leq R}|b_t(x)|^4dt \leq K f(R)^4 n^{-2\theta} \to 0 \notag
\end{align}
as $ n \to \infty $ for every $R$. Also for diffusion and jump coefficients, Assumption B-5 holds trivially.  Thus, Theorem \ref{thm:conv:rc} completes the proof.
\end{proof}

For rate of convergence of scheme \eqref{eq:em:sdewrc}, one takes $\theta=\frac{1}{2}$ in equation \eqref{eq:sde:coeff}.
  \begin{corollary}
Let Assumptions A-3 to  A-6, A-8 and  A-9 be satisfied by the coefficients of SDE given immediately above. Also suppose that Assumptions B-1  and  B-8 hold. Then, the numerical scheme \eqref{eq:em:sdewrc} with coefficients given by \eqref{eq:sde:coeff} achieves the classical rate (of Euler scheme) in $\mathcal{L}^q$ sense i.e.
\begin{align}
E\sup_{t_0 \leq t \leq t_1}|x_{t}-x_{t}^n|^q \leq K n^{-\frac{q}{q+\delta}}
\end{align}
where constant $K>0$ does not depend on $n$.
\end{corollary}
\begin{proof}
By  using equation \eqref{eq:sde:coeff} and Remark \ref{rem:poly:b:rc}, one obtains
\begin{align*}
E\int_{t_0}^{t_1} |b_t^n(x^n_{\kappa(n,t)})-b_t(x^n_{\kappa(n,t)})|^q dt & \leq n^{-2\theta} E\int_{t_0}^{t_1} |b_t(x^n_{\kappa(n,t)})|^{2q}  dt  \leq K n^{-1} (1+E\sup_{t_0 \leq t \leq t_1}|x^n_t|^{q(\chi+2)})
\end{align*}
since $\theta=\frac{1}{2}$. Hence Assumption  B-7 for drift coefficients follows due to Lemma \ref{lem:mb:tes}. For diffusion and jump coefficients, Assumption B-7 holds trivially.  The proof is completed by Theorem \ref{thm:rate:rc:em}.
\end{proof}
\section{Application to Delay Equations}
Let us assume that $\beta_t(y_1, \ldots, y_k, x)$ and $\alpha_t(y_1, \ldots, y_k, x)$ are $\mathscr B([0,T]) \otimes \mathscr B(\mathbb R^{d\times k}) \otimes \mathscr B(\mathbb{R}^d)$-measurable functions and take values in $\mathbb R^d$ and $\mathbb R^{d \times m}$ respectively. Also  let $\lambda_t( y_1, \ldots, y_k, x, z)$ be $\mathscr B([0,T]) \otimes \mathscr B(\mathbb R^{d\times k}) \otimes \mathscr B(\mathbb R^d) \otimes \mathscr{Z}$-measurable function and takes values in $\mathbb R^{d}$. For fixed $H>0$, we consider a $d$-dimensional stochastic delay differential equations (SDDEs)  on $(\Omega, \{\mathscr F_t\}_{t \geq 0}, \mathscr F, P)$ defined  by,
\begin{align} \label{eq:sdde}
dx_t &= \beta_t(y_t, x_t)dt +\alpha_t(y_t, x_t)dw_t + \int_{Z} \lambda_t(y_{t},x_{t},z)\tilde N(dt,dz), \,\,\, t \in [0,T], \notag
\\
x_t& = \xi_t, \,\,\, t \in [-H, 0 ],
\end{align}
where $\xi:[-H,0]\times \Omega\rightarrow\mathbb R^d$ and  $y_t:=(x_{\delta_1(t)}, \ldots, x_{\delta_k(t)})$.  The delay parameters $\delta_1(t), \ldots, \delta_k(t)$ are increasing functions of $t$ and satisfy $-H \leq \delta_j(t) \leq [t/h]h$ for some $h>0$ and $j=1,\ldots, k$.
\begin{remark} \label{rem:T:sdde}
In the following, we assume, without loss of generality, that  $T$ is a  multiple of $h$. If not, then SDDE \eqref{eq:sdde} can be defined for $T' > T$ so that $T'=N' h$, where $N'$ is a positive integer.  The results proved in this article are then recovered for the original SDDE \eqref{eq:sdde} by choosing parameters as $\beta I_{t\leq T}$, $\alpha I_{t \leq T}$ and $\lambda I_{t \leq T}$.
\end{remark}
\begin{remark}
We remark that two popular cases of delay viz. $\delta_i(t)=t-h$ and $\delta_i(t)=[t/h]h$ can be addressed by our findings which have been widely used in literature, for example, \cite{Federico2011,Federico2010, sabanis2011, Oksendal2011} and  references therein.
\end{remark}

\subsection{Existence and Uniqueness}
To prove the existence and uniqueness of the solution of SDDE \eqref{eq:sdde}, we make the following assumptions.
\begin{assumptionC} \label{as:growth:sdde:eu}
For every $R>0$, there exists an $M(R) \in \mathbb{L}^1$ such that
\begin{align}
 x \beta_t(y,x) + |\alpha_t(y,x)|^2 +\int_Z |\lambda_t(y,x,z)|^2 \nu(dz) \leq M_t(R)(1+|x|^2) \notag
\end{align}
for any $t \in [0,T]$  whenever $|y| \leq R$ and $x \in \mathbb{R}^d$.
\end{assumptionC}
\begin{assumptionC} \label{as:local:lipschitz:sdde:eu}
For every $R>0$, there exists an $M(R) \in \mathbb{L}^1$ such that
\begin{align}
(x-\bar{x}) (\beta_t(y, x)-\beta_t(y, \bar{x})) &+ |\alpha_t(y,x)-\alpha_t(y,\bar{x})|^2  + \int_Z |\lambda_t(y, x,z)-\lambda_t(y, \bar{x},z)|^2 \nu(dz) \leq  M_t(R)|x-\bar{x}|^2 \notag
\end{align}
for any $t \in [0,T]$ whenever $|x|,|\bar{x}|, |y| \leq  R$.
\end{assumptionC}
\begin{assumptionC} \label{as:continuity:sdde}
The function $\beta_t(y,x)$ is continuous in $x$ for any $t$ and $y$.
\end{assumptionC}
\begin{theorem} \label{thm:sdde:eu}
Let Assumptions C-1 to C-3 be satisfied. Then there exists a unique solution to SDDE \eqref{eq:sdde}.
\end{theorem}
\begin{proof}
We adopt the approach of  \cite{gs2012} and consider SDDE \eqref{eq:sdde} as a special case of SDE \eqref{eq:sdewrc} by assigning the following values to the coefficients,
\begin{align}
b_t(x)=\beta_t( y_t, x),  \sigma_t(x)=\alpha_t(y_t, x), \gamma_t(x,z)=\lambda_t(y_t, x, z) \label{mkjdnswe}
\end{align}
almost surely for any $t \in [0,T]$. Then the proof is a straightforward generalization of Theorem 2.1 of  \cite{gs2012} and follows due to Theorem \ref{thm:eu:rc}.
\end{proof}
\subsection{Tamed Euler Scheme}
For every $n \in \mathbb{N}$, define the following tamed Euler scheme
\begin{align}
dx_t^n &=  \beta_t^n(y_t^n, x_{\kappa(n,t)}^n) dt + \alpha_t(y_t^n, x_{\kappa(n,t)}^n) dw_t +\int_{Z} \lambda_t( y_t^n, x_{\kappa(n,t)}^n, z)\tilde N(dt,dz),  \,\,\,t \in [0, T], \notag
\\
x_t^n &= \xi_t, \,\,\, t \in [-H, 0], \label{csddeemwjt}
\end{align}
where $y_t^n:=(x_{\delta_1(t)}^n, \ldots, x_{\delta_k(t)}^n)$ and $\kappa$ is defined by \eqref{eq:kappa:sdewrc} with $t_0=0$. Furthermore, for every $n\in \mathbb{N}$,  the drift coefficient is given by
\begin{align}
\beta_t^n(y, x):=\frac{\beta_t(y, x)}{1+n^{-\theta}|\beta_t(y, x)|} \notag
\end{align}
which satisfies
\begin{equation}
|\beta_t^n(y, x)| \leq \min(n^\theta, |\beta_t(y, x)|) \label{ggnhm}
\end{equation}
for any $t \in [0,T]$, $x \in \mathbb{R}^d$ and $y \in \mathbb{R}^{d \times k}$.
\begin{assumptionC} \label{as:initial sdde}
For a fixed $p \geq 2$,  $E\sup_{-H \leq t \leq 0}|\xi_t|^p< \infty$.
\end{assumptionC}
\begin{assumptionC} \label{as: growth sdde}
There exist constants $G >0$ and $\chi \geq 2$ such that
\begin{align}
x \beta_t(y, x)\vee |\alpha_t(y, x)|^2 \vee \int_{Z}  |\lambda_t(y,x,z)|^2\nu(dz)\leq G(1+|y|^\chi+|x|^2) \notag
\end{align}
for any $t \in [0,T]$, $x  \in \mathbb R^d$ and $y \in \mathbb R^{d \times k}$.
\end{assumptionC}
\begin{assumptionC} \label{as: p growth sdde}
There exist constants $G >0$ and $\chi \geq 2$ such that
\begin{align}
 &\int_{Z}  |\lambda_t(y,x,z)|^p\nu(dz) \leq G(1+|y|^{\chi\frac{p}{2}}+|x|^p) \notag
\end{align}
 for any $t \in [0,T]$, $x  \in \mathbb R^d$ and $y \in \mathbb R^{d \times k}$.
\end{assumptionC}
\begin{assumptionC} \label{as:local lipschitz sdde}
For every $R>0$, there exists a constant $K_R >0 $ such that
\begin{align}
(x-\bar x)(\beta_t(y, x)-\beta_t(y, \bar x) ) \vee |\alpha_t(y, x)&-\alpha_t(y, \bar x)|^2  \vee \int_{Z} |\lambda_t(y,x,z)-\lambda_t(y,\bar x,z)|^2 \nu(dz)
\leq K_R |x-\bar x|^2 \notag
\end{align}
for any $t \in [0,T]$ whenever $|x|, |y|, |\bar x| < R$.
\end{assumptionC}
\begin{assumptionC} \label{as:local bound sdde}
For every $R>0$, there exists a constant $K_R>0$ such that
\begin{align}
\sup_{|x| \leq R} \sup_{|y|\leq R} |\beta_t(y, x)|^2 \leq K_R \notag
\end{align}
for any $t \in [0, T]$.
\end{assumptionC}
\begin{assumptionC} \label{as:coeff con sdde}
For every $R>0$ and $t \in [0,T]$,
\begin{align}
\sup_{|x| \leq R}\Big\{|\beta_t(y, x)-&\beta_t(y', x)|^2+|\alpha_t(y, x)-\alpha_t(y', x)|^2  + \int_{Z} |\lambda_t(y, x, z)-\lambda_t(y', x, z)|^2 \nu(dz)\Big\}  \rightarrow 0 \notag
\end{align}
when $y' \rightarrow y$.
\end{assumptionC}
Let us also define,
\begin{equation}
p_i=\Big(\frac{2}{\chi}\Big)^i p \label{eq:p:i}
\end{equation}
for $i=1, \ldots, N'$, where $\chi$ and $p$ satisfy $p/2 \geq (\chi/2)^{N'}$. Also
\begin{align}
p^*=\min_{i}p_i= \Big(\frac{2}{\chi}\Big)^{N'} p. \label{eq:p*}
\end{align}
The following corollary is a consequence of Theorem \ref{thm:conv:rc}.
\begin{corollary}
\label{cmainthmwjt}
Let Assumptions  C-3 to C-9 hold, then
\begin{align*}
\lim_{n \rightarrow \infty}E \sup_{0 \leq t \leq T}|x_t-x_t^n|^q =0
\end{align*}
for any $q  <p^*$.
\end{corollary}
\begin{proof}
First as before, one observes that  SDDE \eqref{eq:sdde} can be regarded as a special case of SDE \eqref{eq:sdewrc} with coefficients given by \eqref{mkjdnswe}. Moreover,  tamed Euler scheme \eqref{csddeemwjt} is a special of  \eqref{eq:em:sdewrc} with coefficients given by
\begin{align} \label{ccoffwjt}
b^n_t(x)=\frac{\beta_t(y^n_t, x)}{1+n^{-\theta}|\beta_t(y^n_t, x)|},   \sigma_t^n(x)=\alpha_t(y^n_t, x), \gamma_t^n(x,z)=\lambda_t(y_t^n, x,z)
\end{align}
almost surely for any $t \in [0, T]$ and $x\in \mathbb R^d$. We shall use inductive arguments to show
\begin{eqnarray}
\lim_{n \rightarrow \infty}E\sup_{(i-1)h \leq t \leq i h} |x_t-x_t^n|^q  =0 \label{cnanut}
\end{eqnarray}
for any $q  < p_{i}$ and for every $i \in \{1, \ldots, N'\}$.
\newline
\textbf{Case: $\mathbf{t \in [0, h]}$.} For  $t \in [0, h]$, one could consider SDDE \eqref{eq:sdde} and their tamed Euler  scheme \eqref{csddeemwjt} as  SDE \eqref{eq:sdewrc} and scheme \eqref{eq:em:sdewrc} respectively with $t_0=0$, $t_1=h$, $x_0=x^n_0=\xi_0$ and with  coefficients given in \eqref{mkjdnswe} and \eqref{ccoffwjt}. Further, one observes that  Assumptions  A-3 to  A-8 and B-1 to B-6 hold due to Assumptions C-3  to C-9. In particular, Assumption A-3 holds due to Assumption C-3 while Assumptions A-4 and B-1 due to Assumption C-4. Further Assumptions A-5, A-6, B-2 and B-3 hold due to  Assumptions C-5 and  C-6 with $L=G$, $M=M_n=1+\Psi^{\chi} \in \mathcal{L}^{\frac{p_1}{2}}$ and $M'=M_n'=1+\Psi^{\chi\frac{p_1}{2}} \in \mathcal{L}^1$, where $\Psi:=\sup_{t \in [0, h]}|(\xi_{\delta_1(t)}, \ldots, \xi_{\delta_k(t)})| \in \mathcal{L}^p$. Also Assumption A-7 holds due to Assumption C-7 with
$$
C_R:=K_R I_{\Omega_R} + \sum_{j=R}^{\infty} K_{j+1} I_{\Omega_{j+1} \backslash{\Omega}_{j}}
$$
where $\Omega_j:=\{\omega \in \Omega:\Psi \leq j \}$. Further one takes $f(R):=K_R$  and then
\begin{align}
P(C_R >f(R)) \leq P(\Psi>R) \leq \frac{E\Psi}{R} \rightarrow 0 \notag
\end{align}
as $R \rightarrow \infty$. This also implies that  Assumption A-8 holds due to Assumption C-8. To verify Assumption B-5, one observes that
\begin{align}
b_t^n(x)=\frac{\beta_t(\xi_{\delta_1(t)}, \ldots, \xi_{\delta_k(t)}, x)}{1+n^{-\theta}|\beta_t(\xi_{\delta_1(t)}, \ldots, \xi_{\delta_k(t)}, x)|} \rightarrow \beta_t(\xi_{\delta_1(t)}, \ldots, \xi_{\delta_k(t)}, x)=b_t(x) \notag
\end{align}
as $n \rightarrow \infty$ and  sequence
\begin{align}
\Big\{ I_{B(R)}\sup_{|x| \leq R}|b_t^n(x)-b_t(x)|^2\Big\}_{\{n \in \mathbb{N}\}} \notag
\end{align}
is uniformly integrable which implies
\begin{align}
\lim_{n \rightarrow \infty} E\int_{t_0}^{t_1}I_{B(R)}\sup_{|x| \leq R}|b_t^n(x)-b_t(x)|^2 dt =0 \notag
\end{align}
and similarly for diffusion and  jump coefficients. Finally Assumption B-6 holds trivially.

Therefore equation \eqref{cnanut} holds due to Theorem \ref{thm:conv:rc}, Lemma \ref{lem:mb:rc} and Lemma \ref{lem:mb:tes} when $i=1$. We note that the convergence here is achieved for all $q < p_1$ and as we proceed to the next interval $[h, 2h]$,  the convergence is achieved in the lower space i.e. $q < p_2$ due to Assumptions C-5 and C-6. Therefore for the inductive arguments, we assume that the convergence in the interval $[(r-1)h, rh]$  is achieved for all $q < p_r$ i.e. we assume that  Theorem \ref{thm:conv:rc}, Lemma \ref{lem:mb:rc} and Lemma  \ref{lem:mb:tes} hold for any $q < p_r$  when $i=r$.
\newline
\textbf{Case: $\mathbf{t \in [r h, (r+1)h]}$.} When $t \in [rh, (r+1)h]$, then SDDE \eqref{eq:sdde} and   scheme \eqref{csddeemwjt} become SDE \eqref{eq:sdewrc} and scheme \eqref{eq:em:sdewrc} respectively with $t_0=rh$, $t_1=(r+1)h$, $x_{t_0}=x_{rh}$, $x^n_{t_0}=x^n_{rh}$ and coefficients given by \eqref{mkjdnswe} and \eqref{ccoffwjt}.
\newline
\textit{Verify  A-3.}  Assumption A-3 holds due to Assumption C-3  trivially.
\newline
\textit{Verify  A-4 and B-1.} Assumptions A-4 and B-1 hold due to Lemma \ref{lem:mb:rc}, Lemma \ref{lem:mb:tes} and inductive assumptions.
\newline
\textit{Verify A-5, A-6, B-2 and B-3.} Assumptions A-5 and B-2  hold due to Assumption C-5 with $M:=1+\sup_{rh \leq t \leq  (r+1)h}|y_t|^\chi$ and $M_n:=1+\sup_{rh \leq t \leq  (r+1)h}|y_t^n|^\chi$ which are bounded in $\mathcal L^\frac{p_{r+1}}{2}$ due to Lemma \ref{lem:mb:rc}, Lemma \ref{lem:mb:tes} and  inductive assumptions. Furthermore Assumptions A-6 and B-3 hold with $M':=1+\sup_{rh \leq t \leq  (r+1)h}|y_t|^{\chi\frac{p_{r+1}}{2}}$ and $M_n':=1+\sup_{rh \leq t \leq  (r+1)h}|y_t^n|^{\chi\frac{p_{r+1}}{2}}$ which are bounded in $\mathcal{L}^1$ due to Lemma \ref{lem:mb:rc}, Lemma \ref{lem:mb:tes}  and  inductive assumptions.
\newline
\textit{Verify  A-7.} For every $R>0$, $|x|, |\bar{x}| \leq R$ and $t \in [rh, (r+1)h]$,  Assumption A-7  holds due to Assumption C-7 with $\mathcal{F}_{rh}$-measurable random variable $C_R$ given by
\begin{align}
C_R:=K_R I_{\Omega_R} + \sum_{j=R}^{\infty} K_{j+1} I_{\Omega_{j+1} \backslash{\Omega}_{j}} \label{jknm1}
\end{align}
where $\Omega_j:=\{\omega \in \Omega:\sup_{t \in [r h, (r+1) h]}|y_t| \leq j \}$. Further one takes $f(R):=K_R$  and then
\begin{align}
P(C_R >f(R)) \leq P\big(\sup_{r h \leq t <(r+1)h}|y_t|>R\big) \rightarrow 0  \, \mbox{as} \,  R \rightarrow \infty.\label{cf*t}
\end{align}
\newline
\textit{Verify A-8.} For every $R>0$ and any $t \in [rh, (r+1)h]$, we take $C_R$ as defined in \eqref{jknm1}, $f(R)=K_R$.  Then one uses \eqref{cf*t} to establish A-8.
\newline
\textit{Verify  B-5.} The inductive assumption implies $|y_t^n-y_t| \rightarrow 0$ in probability and thus due to Assumption C-9, $\sup_{|x|\leq R}|\beta_t^n(y_t^n,x)-\beta_t(y_t,x)| \rightarrow 0$ in probability. Furthermore the sequence
 $$I_{B(R)}\big\{\sup_{|x|\leq R}|\beta_t^n(y_t^n,x)-\beta_t(y_t,x)|^2\big\}_{\{n \in \mathbb{N}\}}$$ is uniformly integrable due to Assumption C-8 and inductive assumption, which implies
\begin{align}
\lim_{n \rightarrow \infty}E \int_{rh}^{(r+1)h} \sup_{|x| \leq R} |b_t^n(x)-b_t(x)|^2 =0. \notag
\end{align}
For  diffusion  coefficient, due to the inductive assumption
$$
\Big\{\sup_{rh \leq t \leq (r+1)h}|y_t^n-y_t|^\chi\Big\}_{\{n \in \mathbb{N}\}} \, \mbox{and hence} \, \,  \Big\{\sup_{rh \leq t \leq (r+1)h}|y_t^n|^\chi\Big\}_{\{n \in \mathbb{N}\}}
$$
are uniformly integrable which on using Assumptions C-5 to C-6 imply
$$
\Big\{\sup_{|x| \leq R}|\alpha_t(y_t^n,x)-\alpha_t(y_t,x)|^2\Big\}_{\{n \in \mathbb{N}\}}
$$
is uniformly integrable. Moreover due to Assumption C-9,
$$
\sup_{|x| \leq R}|\alpha_t(y_t^n,x)-\alpha_t(y_t,x)|^2 \rightarrow 0
$$
in probability as $n \rightarrow \infty$ and therefore Assumption B-5 holds for the diffusion coefficients. One adopts similar arguments for jump coefficients.
\newline
\textit{Verify  B-6.} This follows due to the inductive assumptions.
\newline
This completes the proof.
\end{proof}

We now proceed to obtain the rate of convergence of the scheme \eqref{csddeemwjt}. For this purpose, we replace Assumptions C-7 and C-9 by the following assumptions.
\begin{assumptionC} \label{as:lipschitz:sdde}
 There exist constants $C>0$, $q \geq 2$ and  $\chi>0$ such that,
\begin{align}
 (x-\bar{x})(\beta_t(y,  x)-\beta_t(  y, \bar{x})) \vee|\alpha_t( y, x)-\alpha_t( y, \bar{x})|^2 \vee & \notag
\\
  \int_{Z} |\lambda_t( y, x, z)-\lambda_t( y, \bar{x}, z)|^2 \nu(dz) &\leq C|x-\bar{x}|^2 \notag
\\
\int_{Z} |\lambda_t( y, x, z)-\lambda_t( y, \bar{x}, z)|^q \nu(dz)& \leq C|x-\bar{x}|^q \notag
\\
|\beta_t(  y,  x)-\beta_t(  y,  \bar{x}) |^2  &\leq  C( 1+|x|^{\chi}+|\bar{x}|^{\chi}) |x-\bar{x}|^2 \notag
\end{align}
for any $t \in [0, T]$, $ x , \bar{x} \in \mathbb R^d$,  $y  \in \mathbb R^{d \times k}$ and a $\delta \in (0,1)$ such that $\max\{ (\chi+2)q, \frac{q \chi}{2}\frac{q+\delta}{\delta}\}\leq  p^*$.
\end{assumptionC}

\begin{assumptionC} \label{as:poly:lipschitz:sdde}
Assume that
\begin{align}
 |\beta_t(  y,  x)-\beta_t(  \bar{y},  x) |^2  & \vee |\alpha_t(  y,  x)-\alpha_t(  \bar{y},  x) |^2 \vee \Big(\int_Z |\lambda_t(  y,  x)-\lambda_t(  \bar{y},  x) |^\zeta \nu(dz)\Big)^\frac{q}{\zeta}  \leq  C(1+|y|^{\chi}+|\bar{y}|^{\chi})| y-\bar{y}|^2 \notag
\end{align}
where $\zeta=2,q$, for any $t \in [0,T]$,  $ x \in \mathbb R^d$ and $y , \bar{y} \in \mathbb R^{d \times k}$.
\end{assumptionC}
\begin{remark} \label{as:poly:growth:beta:sdde}
 Due to Assumptions C-8, C-10 and C-11, there exists a constant $C>0$ such that
\begin{align}
|\beta_t(y,x)|^2 \leq C(1+|y|^{\chi+2}+|x|^{\chi+2}) \notag
\end{align}
for  any $t \in[0,T]$, $x \in \mathbb R^d$ and $y  \in \mathbb R^{d \times k}$.
\end{remark}

In the following corollary, we obtain a convergence rate for the tamed Euler scheme \eqref{csddeemwjt} which is equal to the classical convergence rate of Euler scheme. For this purpose, one can take $\theta=\frac{1}{2}$.
\begin{corollary} \label{thm:rate:tem:sdde}
Let Assumptions C-3 to C-6, C-8, C-10 and C-11 be satisfied. Then
\begin{align} \label{eq:rate:i}
E \sup_{0 \leq t \leq T}|x_t-x_t^n|^q \leq K n^{-\frac{q}{q+N'\delta}}
\end{align}
for any $q  <p^*$ where constant $K>0$ does not depend on $n$.
\end{corollary}
\begin{proof}
The corollary can be proved by adopting similar arguments as used in the proof of Corollary \ref{cmainthmwjt}. For this purpose, one can use Theorem \ref{thm:rate:rc:em} inductively to show that for every $i=1,\ldots,N'$,
\begin{align*}
E \sup_{(i-1)h \leq t \leq ih}|x_t-x_t^n|^q \leq K n^{-\frac{q}{q+i\delta}}
\end{align*}
for any $q  <p_i$ where constant $K>0$ does not depend on $n$. Now, notice that Assumptions A-3 through A-6, A-8 and B-1 through B-3 have already been verified in the proof of Corollary \ref{cmainthmwjt}. Hence, one only needs to verify Assumptions A-9, B-7 and B-8.
\newline
\textbf{Case} $\mathbf{t \in [0,h]}.$ As before, one considers SDDE \eqref{eq:sdde} as a special case of SDE \eqref{eq:sdewrc} with $t_0=0$, $t_1=h$, $x_{t_0}=\xi_0$ and coefficients given by equation \eqref{mkjdnswe}. Also, scheme \eqref{csddeemwjt} can
be considered as a special case of scheme \eqref{eq:em:sdewrc} with $t_0=0$, $t_1=h$, $x_{t_0}=\xi_0$ and coefficients given by \eqref{ccoffwjt}.
\newline
\textit{Verify A-9.} Assumption A-9 follows from Assumption C-10 trivially.
\newline
\textit{Verify B-7.} Notice that $y_t=y_t^n=:\Phi_t$ for $t \in [0,h]$ which implies
\begin{align*}
 E\int_{0}^{h}|b^n_t(x_{\kappa(n,t)}^n)-b_t(x_{\kappa(n,t)}^n)|^q dt  \leq n^{-q \theta} E\int_{0}^{h}\big|\beta_t(\Phi_t, x_{\kappa(n,t)}^n)\big|^{2q} dt
\end{align*}
which on using Remark \ref{as:poly:growth:beta:sdde}, Assumption C-4 and Lemma \ref{lem:mb:tes} gives
\begin{align*}
 E\int_{0}^{h}|b^n_t(x_{\kappa(n,t)}^n)-b_t(x_{\kappa(n,t)}^n)|^q dt   & \leq n^{-q \theta} K \big(1+ E\Psi^{(\chi+2) q}
 + E\sup_{0 \leq t \leq h}| x_{\kappa(n,t)}^n|^{(\chi+2) q}\big) \leq K n^{-\frac{q}{2}}
\end{align*}
for any $q<p_1$ because $\theta=\frac{1}{2}$.
\newline
\textit{Verify B-8.} This holds trivially.
\newline
Thus, by Theorem \ref{thm:rate:rc:em}, one obtains that equation \eqref{eq:rate:i} holds for $i=1$. For inductive arguments, one assumes that equation \eqref{eq:rate:i} holds for $i=r$ and then verifies it for $i=1+r$.
\newline
\textbf{Case} $\mathbf{t \in [rh,(r+1)h]}.$ Again, consider SDDE \eqref{eq:sdde}  as a special case of SDE \eqref{eq:sdewrc} with $t_0=rh$, $t_1=(r+1)h$, $x_{t_0}=x_{rh}$ and coefficients given by equation \eqref{mkjdnswe}. Similarly, consider scheme \eqref{csddeemwjt}  as a special case of scheme \eqref{eq:em:sdewrc} with $t_0=rh$, $t_1=(r+1)h$, $x_{t_0}=x_{rh}$ and  coefficients given by  \eqref{ccoffwjt}.
\newline
\textit{Verify A-9.} Assumption A-9 follows from Assumption C-10 trivially.
\newline
\textit{Verify B-7.} One observes that
\begin{align*}
& \qquad\qquad  E\int_{0}^{h}|b^n_t(x_{\kappa(n,t)}^n)-b_t(x_{\kappa(n,t)}^n)|^q dt
\\
&\leq K E\int_{0}^{h}\big|\frac{\beta_t(y_t^n, x_{\kappa(n,t)}^n)}{1+n^{-\theta}|\beta_t(y_t^n, x_{\kappa(n,t)}^n)|}-\beta_t(y_t^n, x_{\kappa(n,t)}^n)\big|^q dt
\\
& + K E\int_{0}^{h} \big|\beta_t(y_t^n, x_{\kappa(n,t)}^n)-\beta_t(y_t, x_{\kappa(n,t)}^n)\big|^q dt
\\
& \leq K  n^{-q\theta}E\int_{0}^{h}|\beta_t(y_t^n,x_{\kappa(n,t)}^n)|^{2q} dt + K E\int_{0}^{h} \big|\beta_t(y_t^n, x_{\kappa(n,t)}^n)-\beta_t(y_t,x_{\kappa(n,t)}^n)\big|^q dt
\end{align*}
which on the application of Remark \ref{as:poly:growth:beta:sdde} and Assumption C-11 gives
\begin{align*}
E\int_{0}^{h}&|b^n_t(x_{\kappa(n,t)}^n)-b_t(x_{\kappa(n,t)}^n)|^q dt \leq  K n^{-q\theta} E\int_{0}^{h} (1+|y_t^n|^{(\chi+2)q}+|x_{\kappa(n,t)}^n|^{(\chi+2)q}) dt
\\
&+ K E\int_{0}^{h} (1+|y_t|^\frac{q\chi}{2}+|y_t^n|^\frac{q\chi}{2}) |y_t-y_t^n|^q dt
\end{align*}
and then on the application of H\"older's inequality, Lemma \ref{lem:mb:rc} and Lemma \ref{lem:mb:tes} along with inductive assumptions gives
\begin{align*}
E\int_{0}^{h}|b^n_t(x_{\kappa(n,t)}^n)-b_t(x_{\kappa(n,t)}^n)|^q dt \leq K n^{-q\theta} + K E\int_{0}^{h} (E|y_t-y_t^n|^{q+\delta})^\frac{q}{q+\delta}.
\end{align*}
Finally on using the inductive assumption and $\theta=\frac{1}{2}$, one obtains
\begin{align*}
E\int_{0}^{h}|b^n_t(x_{\kappa(n,t)}^n)-b_t(x_{\kappa(n,t)}^n)|^q dt \leq K n^{-\frac{q}{2}} + K n^{-\frac{q}{q+(r+1)\delta}}
\end{align*}
and hence \eqref{eq:rate:i} holds for $i=r+1$.
\newline
\textit{Verify B-8.} This holds due to inductive assumptions.
\newline
Thus, by Theorem \ref{thm:rate:rc:em}, one obtains that equation \eqref{eq:rate:i} holds for $i=r+1$. This completes the proof.
\end{proof}

\section{Numerical Illustrations}
We demonstrate our results numerically with the help of following examples.
\newline
\textbf{Example 1.} Consider the following SDE,
\begin{align} \label{eq:sde:sim}
dx_t =& -x_t^5 dt + x_t dw_t + \int_{\mathbb{R}} x_t z \tilde N(dt, dz)
\end{align}
for any $t \in [0, 1]$  with initial value $x_0 =1$. The jump size follows standard normal distribution and jump intensity is $3$. The tamed Euler scheme with step-size $2^{-21}$ is taken as true solution. Table \ref{tab:sde} and Figure \ref{fig:sde} are based on 1000 simulations.
\begin{table}[h]
\footnotesize
\centering
\footnotesize
\caption{\em SDE: Errors in the tamed Euler scheme.}
\begin{tabular}{|c|c|c|} \hline
step-size& $\sqrt{E|x_t-x_t^n|^2}$ & $E|x_t-x_t^n|$  \\ \hline
$2^{-20}$& 0.000983465083412957  &	0.000359729516674718 \\
$2^{-19}$& 0.00216716723504906	  & 0.000696592563650715 \\
$2^{-18}$& 0.00392575778408420	  & 0.00117629823362591 \\
$2^{-17}$& 0.00577090918102760	  & 0.00176826651345228 \\
$2^{-16}$& 0.00788070333470230	  & 0.00265746428431957 \\
$2^{-15}$& 0.0114588451477506    & 0.00398287796962204 \\
$2^{-14}$& 0.0152592153162732	  & 0.00568182096844841 \\
$2^{-13}$& 0.0214987425830999    & 0.00775473960140893 \\
$2^{-12}$& 0.0300412202466655	  & 0.0117456051149168 \\
$2^{-11}$& 0.0434809466351964	  & 0.0168998838844189 \\ \hline
\end{tabular}
 \label{tab:sde}
\end{table}
\newline
\noindent
\textbf{Example 2.} Consider the following SDDE,
\begin{align} \label{eq:sdde:sim}
dx_t = (x_t-x_t^3+ y_t^2)dt+(x_t+y_t^3)dw_t+ \int_{\mathbb{R}} (x_t+y_t) z \tilde N(dt,dz)
\end{align}
where $y_t=x_{t-1}$ for $t\in [0,2]$ with initial data $\xi_t=t+1$ for $t \in [-1,0]$.
\begin{figure}[!ht]
        \centering
        \begin{subfigure}[b]{0.49\linewidth} \centering
                \includegraphics[width=\linewidth]{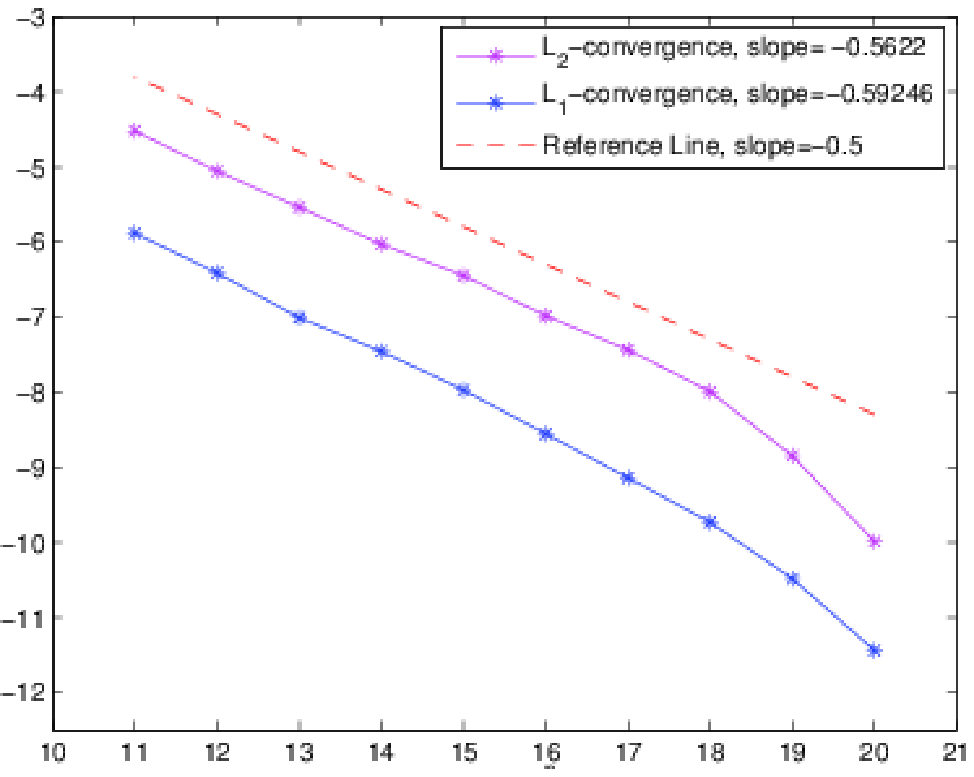}
                \caption{\em SDE: $\mathcal{L}^1$ and $\mathcal{L}^2$ convergence with rate}
                \label{fig:sde}
        \end{subfigure} \hfill
        \begin{subfigure}[b]{0.49\linewidth} \centering
                \includegraphics[width=\linewidth]{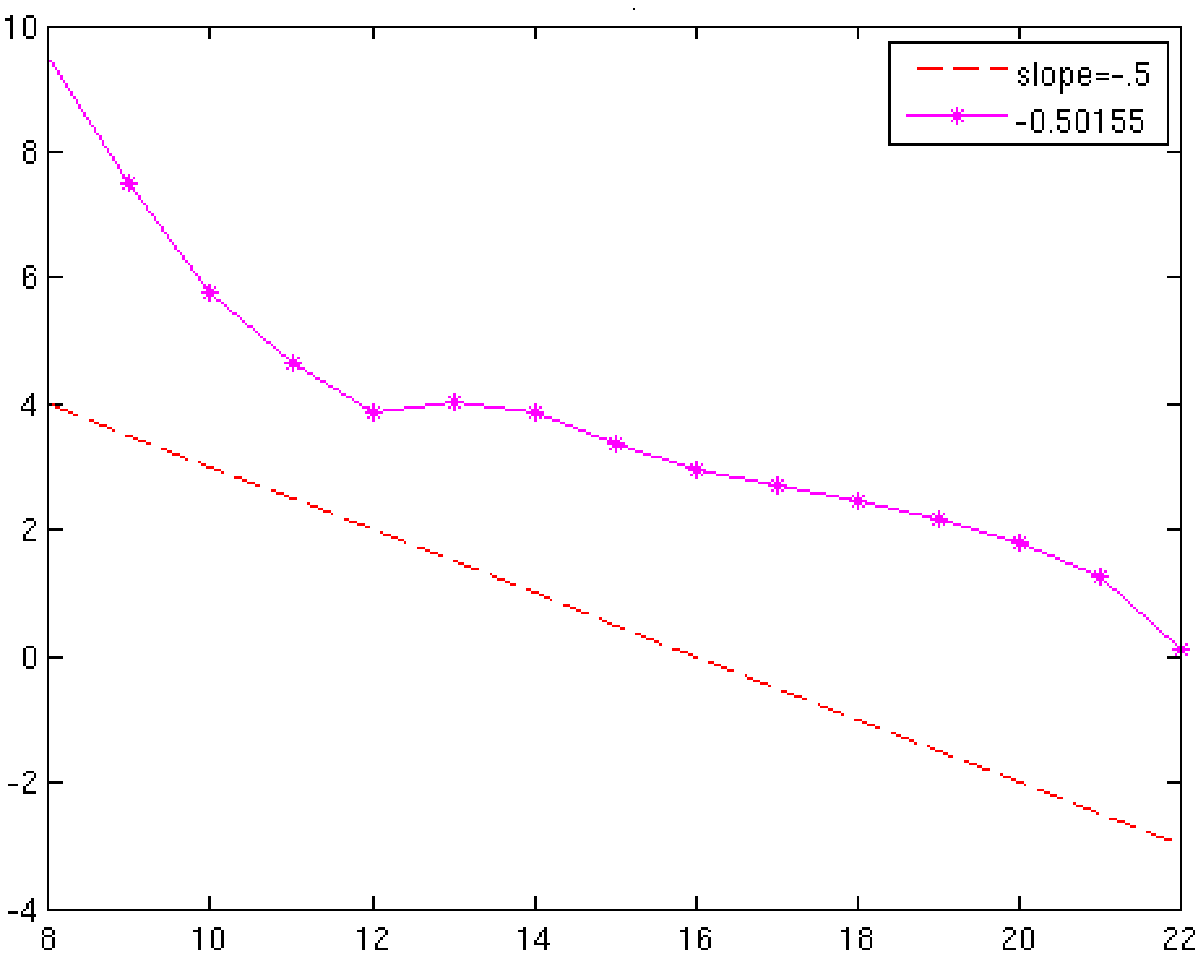}
                \caption{\em SDDE: $\mathcal{L}^2$ convergence with rate}
                \label{fig:sdde}
        \end{subfigure}
        \caption{\em Tamed Euler Schemes of SDE \eqref{eq:sde:sim} and SDDE \eqref{eq:sdde:sim}}\label{fig:}
\end{figure}
The jump size follows standard Normal distribution and jump intensity  is $3$. The tamed scheme with step size $2^{-23}$ is taken as the true solution. Figure \ref{fig:sdde} is based on $300$ sample paths.

\end{document}